\documentclass[11pt,a4paper]{article}
\usepackage[toc,page]{appendix}

\usepackage[all,cmtip]{xy}
\entrymodifiers={+!!<0pt,\fontdimen22\textfont2>}

\usepackage{fouriernc}
\usepackage[T1]{fontenc}
\usepackage[english]{babel}         
\usepackage{verbatim}               

\usepackage{amsmath}
\usepackage{amsfonts}
\usepackage{amssymb}
\usepackage{mathrsfs}    
\usepackage{amsthm}
\usepackage{xcolor}

\usepackage{hyperref}           

\usepackage{cite}
\usepackage{stmaryrd}

\usepackage{graphicx}
\usepackage{pgf}
\usepackage{eepic}
\usepackage{pstricks}
\usepackage[all,cmtip]{xy}

\newcommand{\ber}{\mathrm{Ber}}
\newcommand{\ad}{\mathrm{ad}}
\newcommand{\spa}[1]{\mathrm{Spa}\left(#1\right)}
\newcommand{\spf}[1]{\mathrm{Spf}\left(#1\right)}

\newcommand{\weak}[1]{\langle #1\rangle^\dagger}
\newcommand{\et}{\mathrm{\'{e}t}}

\usepackage{epsfig}   

\theoremstyle{plain}
\newtheorem*{thrm}{Theorem}
\newtheorem{theo}{Theorem}[section]

\newtheorem{prop}[theo]{Proposition}
\newtheorem{cor}[theo]{Corollary}

\newtheorem{lem}[theo]{Lemma}

\newtheorem{guess[theo]}{Guess}
\theoremstyle{definition}

\newtheorem{defn}[theo]{Definition}

\theoremstyle{remark}
\newtheorem{rem}[theo]{Remark}
\newtheorem{exa}[theo]{Example}
\newtheorem*{note}{Note}

\newcommand{\proofof}[1]{\end{#1}\begin{proof}}

\makeatletter
\renewcommand\section{\@startsection {section}{1}{\z@}%
  {-3.5ex \@plus -1ex \@minus -.2ex}{2.3ex \@plus.2ex}%
  {\normalfont\large\bfseries}}
\renewcommand\subsection{\@startsection{subsection}{2}{\z@}%
  {-3.25ex\@plus -1ex \@minus -.2ex}{1.5ex \@plus .2ex}%
  {\normalfont\bfseries}}

\newcommand{\sh}[1]{\mathcal{#1}}

\newcommand{\Q}{{\mathbb Q}}
\newcommand{\R}{{\mathbb R}}

\renewcommand{\P}{{\mathbb P}}
\newcommand{\A}{{\mathbb A}}

\DeclareMathAlphabet{\mathrmsl}{OT1}{cmr}{m}{sl}

\newcommand{\rssymb}[2]{\newcommand{#1}{\mathrmsl{#2}} }
\newcommand{\oper}[3][n]{\newcommand{#2}{\mathop{\mathrm{#3}}%
\ifx n#1\nolimits\else\limits\fi} }
\newcommand{\rsoper}[3][n]{\newcommand{#2}{\mathop{\mathrmsl{#3}}%
\ifx n#1\nolimits\else\limits\fi} }

\newcommand{\lser}[1]{(\!(#1)\!)}
\newcommand{\pow}[1]{\llbracket #1 \rrbracket}

\newcommand{\spec}[1]{\mathrm{Spec}\left(#1\right)}
\newcommand{\cur}[1]{\mathcal{#1}}
\newcommand{\Norm}[1]{\left\Vert #1\right\Vert}
\newcommand{\norm}[1]{\left\vert#1\right\vert}
\newcommand{\isomto}{\overset{\sim}{\rightarrow}}
\newcommand{\bu}{\bullet}

\newcommand{\rig}{\mathrm{rig}}
\newcommand{\ek}{\cur{E}_K}
\newcommand{\ekd}{\cur{E}_K^\dagger}
\newcommand{\tate}[1]{\langle #1 \rangle}

\oper\Ad{Ad}

\oper\val{val}
\oper\coker{coker}
\oper\mult{mult}
\oper\Iso{Iso}
\oper\End{End}
\oper\Aut{Aut}
\oper\Sub{Sub}
\oper\Alt{Alt}
\oper\Ext{Ext}
\oper\Pic {Pic}
\oper\Sym{Sym}
\oper\Spec{Spec}
\oper\Spf{Spf}
\oper\Sp{Sp}
\oper\Spa{Spa}
\oper\Proj{Proj}
\rsoper\divg{div}
\rsoper{\sym}{sym}
\rsoper{\alt}{alt}
\rsoper\trace{tr}
\rssymb\id{id}

\newcommand{\thismonth}{\ifcase\month\or
  January\or February\or March\or April\or May\or June\or
  July\or August\or September\or October\or November\or December\fi
  \space\number\year}

\title{Rigid cohomology over Laurent series fields I: First definitions and basic properties}
\author{Christopher Lazda and Ambrus P\'{a}l}

\begin{document}

\maketitle 

\abstract{This is the first in a series of papers in which we construct and study a new $p$-adic cohomology theory for varieties over Laurent series fields $k\lser{t}$ in characteristic $p$. This will be a version of rigid cohomology, taking values in the bounded Robba ring $\ekd$, and in this paper, we give the basic definitions and constructions. The cohomology theory we define can be viewed as a relative version of Berthelot's rigid cohomology, and is constructed by compactifying $k\lser{t}$-varieties as schemes over $k\pow{t}$ rather than over $k\lser{t}$. We reprove the foundational results necessary in our new context to show that the theory is well defined and functorial, and we also introduce a category of `twisted' coefficients. In latter papers we will show some basic structural properties of this theory, as well as discussing some arithmetic applications including the weight monodromy conjecture and independence of $\ell$ results for equicharacteristic local fields.}

\tableofcontents

\section*{Introduction}\addcontentsline{toc}{section}{Introduction}

This is the first in a series of papers \cite{rclsf2,rclsf3} dedicated to the construction of a new $p$-adic cohomology theory for varieties over local fields of positive characteristic, that is fields which are isomorphic to a Laurent series field $k\lser{t}$ over a finite field $k$.   (Actually we expect to be able to phrase things purely in terms of valued fields in characteristic $p$, but for now we stick to the explicit case of $k\lser{t}$ for simplicity).   Here we will give a general overview of the context and motivation behind the work, and provide an introduction to the series of papers as a whole.

The study of the cohomology of arithmetic varieties has a long and distinguished history, arguably beginning with Weil's address 1954 ICM address in which he speculated that a suitably robust cohomology theory for varieties over finite fields would have remarkable implications concerning the numbers of points on varieties over finite fields. Ever since then, the search for and study of such `Weil cohomology theories' has been a major driving force in algebraic geometry and number theory, the clearest example of this being the foundational work of Grothendieck and his school on the \'{e}tale topology for schemes, culminating in Deligne's final proof of the Riemann hypothesis for smooth projective varieties of all dimensions in his 1974 paper.

Even though it was the $\ell$-adic theory that eventually led to the first full proof of the Weil conjectures, rationality of the zeta function was first proved via $p$-adic methods by Dwork, a full 10 years before Grothendieck's proof via \'{e}tale cohomology. Dwork's method was later put into a more conceptual framework of a $p$-adic cohomology theory by Monsky and Washnitzer, and thanks largely to work of Berthelot, this eventually grew into the more sophisticated theory of rigid cohomology. Since the early success of $p$-adic approaches, however, it is the $\ell$-adic theory that has generally lead the way in known results and power of the machinery, for example, finite dimensionality of $p$-adic cohomology for smooth varieties was not known until Berthlot's proof in 1997, and it was not known that $p$-adic cohomology admitted Grothendieck's `six cohomological operations' until recent work of Caro and Tsuzuki. A completely $p$-adic proof of the full Weil conjectures was only given in 2006 by Kedlaya. 

One particular area in which we are interested in is the cohomology of varieties over local fields, and in particular the interaction with the cohomology of their reductions over the residue field. So let us suppose that we have a smooth and proper variety $X$ over a local field $F$ of residue characteristic $p>0$, and let $\ell$ be a prime different from $p$. Then, like all $\ell$-adic representations, the \'{e}tale cohomology $H^i_\et(X_{F^\mathrm{sep}},\Q_\ell)$ is quasi-unipotent, that is after making a finite separable extension of $F$ the inertia group acts unipotently. This is a cohomological interpretation of semistable reduction, geometrically speaking we expect that there exists a finite separable extension of $F$ such that $X$ admits a semistable model over $\cur{O}_F$, the ring of integers of $F$.

In the case where we do actually have semistable reduction, i.e. a semistable scheme $\sh{X}/\cur{O}_F$ with generic fibre $X$, then there is a close relation between the cohomology of $X$ and that of the special fibre $\sh{X}_s$ - this is given by the weight spectral sequence. 
$$ \bigoplus_{l\geq \max\left\{ 0,-r\right\}} H^{i-r-2l}_\et(D_{\bar{k}}^{(2l+r+1)},\Q_\ell(-r-l))\Rightarrow H^i_\et\left(X_{F^\mathrm{sep}},\Q_\ell\right)
$$
where the $D^{(l)}$ are disjoint unions of intersections of the components of $\sh{X}_s$. In this case, the weight monodromy conjecture asserts that the filtration induced by the weight spectral sequence on $ H^i_\et(X_{F^\mathrm{sep}},\Q_\ell)$ is equal to the monodromy filtration, coming from the quasi-unipotence of the action of the interia group $I_F$. This is closely bound up with the notion of $\ell$-independence, which more or less states that the Galois representation $H^i_\et(X_{F^\mathrm{sep}},\Q_\ell)$ is independent of $\ell$. More specifically, the $\ell$-adic monodromy theorem allows us to attach Weil--Deligne representations to the \'{e}tale cohomology of some variety $X$ over $F$, and the one conjectures that the whole family of $\ell$-adic Weil--Deligne representations for $\ell\neq p$ is `compatible' in a certain precise sense.

When $F$ is of mixed characteristic $(0,p)$ and $\ell=p$ then the story is more complicated. For example, one no longer has a monodromy theorem, and one has to impose a condition of `potential semistability' to get a reasonable category of Galois representations, it is then a hard theorem that all representations coming from geometry are potentially semistable. As in the $\ell$-adic situation, it is this potential semistability that is crucial in attaching Weil--Deligne representations to the $p$-adic Galois representations coming from geometry, and thus being able to formulate a weight-monodromy conjecture in this case, as well as including the case $\ell=p$ in $\ell$-independence results and conjectures. The vanishing cycles formalism is also a lot more involved, see for example \cite{geiss,tsuj}.

The situation we are interested in, namely $F$ of equal characteristic $p$ and $\ell=p$ (i.e. $F\cong k\lser{t}$) is even more mysterious. More precisely let $k$ be a field of characteristic $p$, $\cur{V}$ a complete DVR whose residue field is $k$ and fraction field $K$ is of characteristic $0$. Let $\pi$ be a uniformiser for $\cur{V}$, fix a norm $\norm{\cdot}$ on $K$ such that $\norm{p}=p^{-1}$, and let $r=\norm{\pi}^{-1}>1$. Here the monodromy theorem one has concerns $(\varphi,\nabla)$-modules over the Robba ring $\cur{R}_K$, where the Robba ring 
$$ \cur{R}_K = \left\{\left.\sum_i a_it^i \in K\pow{t,t^{-1}}\;\right|\; \begin{matrix} \exists \eta<1 \text{ s.t. } \norm{a_i}\eta^i\rightarrow 0\text{ as }i\rightarrow -\infty \\ \forall \rho<1,\;\norm{a_i}\rho^i\rightarrow 0\text{ as }i\rightarrow \infty\end{matrix}\right\}
$$
is the ring of functions converging on some open annulus $\eta<\norm{t}<1$. These $(\varphi,\nabla)$-modules can then be viewed as $p$-adic differential equations on an open annulus, together with a Frobenius structure (here $K$ is a complete $p$-adic field of characteristic $0$ with residue field $k$). As well as the `constant' $(\varphi,\nabla)$-module $\cur{R}_K$ and its Tate twists, examples of such modules include the Bessel isocrystal introduced by Dwork in \cite{dw}. This is a free $\cur{R}_K$-module of rank $2$, with basis $\{e_1,e_2\}$ on which the connection acts via $$\nabla(e_1)=t^{-2}\pi^2e_2 ,\;\;\nabla(e_2)=t^{-1}e_1$$
and the Frobenius acts via some matrix  $A$ with $\det A=p$ and $A\equiv \left( \begin{matrix}1 & 0 \\ 0 & 0 \end{matrix} \right) \;(\pi)$. 

These modules appear in nature as the fibres of overconvergent $F$-isocrystals on a smooth curve $C/k$ around a missing point
$$ \spec{k\lser{t}} \rightarrow C
$$
and can be viewed as a $p$-adic analogue of $\ell$-adic Galois representations.

However, has previously been no satisfactory link between these objects and the $p$-adic cohomology of varieties over $k\lser{t}$, which is currently most easily expressed as taking values in the category of $(\varphi,\nabla)$-modules over a related ring, the Amice ring
$$ \ek = \left\{\left.\sum_i a_it^i \in K\pow{t,t^{-1}}\;\right|\;  \sup_i\norm{a_i}<\infty,\;a_i\rightarrow 0\text{ as }i\rightarrow -\infty \right\}
$$
which is a complete $p$-adic field with residue field $k\lser{t}$, and the theory that gives rise to these cohomology groups is Berthelot's rigid cohomology $X\mapsto H^i_\rig(X/\ek)$. Here $(\varphi,\nabla)$-modules have no obvious geometric interpretation, since $\ek$ itself does not, but they still make sense as finite dimensional vector spaces over $\ek$ equipped with a connection and a horizontal Frobenius. Although $\ek$ and $\cur{R}_K$ are both subrings of the double power series ring $K\pow{t,t^{-1}}$, neither of them contains the other, and thus there is no straightforward way to relate $(\varphi,\nabla)$-modules over either to those over the other, and hence no straightforward way in which to view quasi-unipotence of $(\varphi,\nabla)$-modules over $\cur{R}_K$ as a cohomological manifestation of potential semistability.  

One of our goals in constructing a new theory of $p$-adic cohomology for varieties over $k\lser{t}$ is to connect rigid cohomology to the quasi-unipotence theorem, by showing that each $H^i_\rig(X/\ek)$ has a canonical lattice over the bounded Robba ring
$$ \ekd= \left\{\left.\sum_i a_it^i \in K\pow{t,t^{-1}}\;\right|\;  \sup_i\norm{a_i}<\infty,\;\exists \eta<1 \text{ s.t. } \norm{a_i}\eta^i\rightarrow 0\text{ as }i\rightarrow -\infty \right\} 
$$
which appears as the intersection of $\cur{R}_K$ and $\ek$ inside $K\pow{t,t^{-1}}$. This is a Henselian valued field with residue field $k\lser{t}$, and what we expect is a `refinement' of rigid cohomology 
$$ X\mapsto H^*_\rig(X/\ekd)
$$
takings values in vector spaces over $\ekd$, such that when we base change to $\ek$ we recover $\ek$-valued rigid cohomology. These spaces should also come with the structure of $(\varphi,\nabla)$-modules over $\ekd$. It is also worth noting that those $(\varphi,\nabla)$-modules over $\cur{R}_K$ that arise in geometry as the fibres of overconvergent isocrystals on smooth curves around missing points are actually canonically defined over $\ekd$, since the pullback functor actually factors naturally through the category of $(\varphi,\nabla)$-modules over $\ekd$. 

Actually there is a sense in which the existence of the theory $H^*_\rig(X/\ekd)$ can be viewed as an analogue of the mixed characteristic result that $p$-adic Galois representations coming from geometry are potentially semistable. The natural base extension functor from $(\varphi,\nabla)$-modules over $\ekd$ to those over $\ek$ is fully faithful, thus there is a natural condition on $(\varphi,\nabla)$-modules of being `overconvergent', that is of coming from a $(\varphi,\nabla)$-module over $\ekd$. It is exactly those `overconvergent' modules which can be base changed to $\cur{R}_K$, and thus can be said to be `quasi-unipotent' in a certain sense. We can therefore view overconvergence as an equicharacteristic analogue of potential semistability, and the existence of an $\ekd$-valued cohomology theory as proving that those $(\varphi,\nabla)$-modules coming from geometry are overconvergent.

When we simply take $\cur{V}=W=W(k)$ to be the Witt vectors of $k$, then, for smooth and proper varieties at least, this `overconvergence' property of $H^*_\rig(X/\ek)$ was proven by Kedlaya in his thesis \cite{kedthesis}, and using full faithfulness this means that one can just simply define $H^*_\rig(X/\ekd)$ to be any $(\varphi,\nabla)$-module over $\ekd$ whose base change to $\ek$ is $H^*_\rig(X/\ek)$, and one gets a functor on the category of smooth and proper $k\lser{t}$-varieties essentially for free. Depsite this, there are many justifications for the long and sometimes tedious effort of setting up a new theory of rigid cohomology that we undertake in these papers. Firstly, it is always conceptually and practically more satisfying to actually construct something than merely show that it has to exist - Kedlaya's overconvergence result shows that the groups $H^*_\rig(X/\ekd)$ have to exist, at least for smooth and proper varieties, however, what we show is how we expect these cohomology groups to be constructed. Secondly, the fact that our construction applies to arbitrary varieties over $k\lser{t}$, not necessarily smooth or proper, will enable us to extend these results to include open and/or singular varieties. Finally, and currently somewhat more speculatively, we expect that our approach to the problem will most naturally lead to the development of a $p$-adic vanishing cycles formalism and the construction of a $p$-adic weight spectral sequence describing the weight filtration on $H^i_\rig(X/\ekd)\otimes_{\ekd} \cur{R}_K$, at least once a suitably robust cohomological formalism has been developed.

To reiterate then, what we are seeking therefore is a refinement of Berthelot's theory of rigid cohomology taking values in vector spaces over the bounded Robba ring $\ekd$, and the clue as to how to proceed in our construction comes from the observation that $\ekd$ itself can be viewed as a kind of `dagger algebra' over $K$, of the sort that appears in Monsky-Washnitzer or rigid cohomology. Dagger algebras are quotients of the ring of overconvergent power series
$$ K\weak{x_1,\ldots,x_n}= \left\{\left.\sum_Ia_Ix^I \right\vert \exists \rho>1\text{ s.t. } \norm{a_I}\rho^I\rightarrow 0 \right\}
$$
and appear when one calculates the rigid cohomology of smooth affine varieties $X=\spec{A_0}/k$: one takes a dagger algebra $A$ lifting $A_0$, then the rigid cohomology of $X$ is just the de Rham cohomology of $A$.

If we let $S_K=\cur{V}\pow{t}\otimes_\cur{V} K$, and equip it with the $\pi$-adic topology, then the Amice ring $\ek$ is just the `Tate algebra' $$S_K\tate{t^{-1}}=\frac{S_K\tate{x}}{(tx-1)}$$ over $S_K$, and the bounded Robba ring $\ekd$ is the corresponding `dagger algebra' $$S_K\weak{t^{-1}}=\frac{S_K\weak{x}}{(tx-1)}.$$
In more geometric terms, we consider the `frame' 
$$\left(\spec{k\lser{t}},\spec{k\pow{t}},\spf{\cur{V}\pow{t}}\right)$$
and let us suppose for a minute that we have a good notion of the `generic fibre' of $\spf{\cur{V}\pow{t}}$, as some kind of rigid space over $K$, let us call it $\mathbb{D}^b_K$ (the notation is meant to suggest some form of `bounded open unit disc', or an `open unit disc with infinitesimal boundary'). If we believe that Berthelot's notions of tubes and strict neighbourhoods can also made to be work in this context, then the tube of $\spec{k\lser{t}}$ inside $\mathbb{D}^b_K$ should be defined by
$$ \left\{\left.x\in \mathbb{D}^b_K \right|\norm{t(x)}\geq 1\right\}$$
and a cofinal system of `strict' neighbourhoods of this tube inside $\mathbb{D}^b_K$ should be given by
$$ \left\{\left.x\in \mathbb{D}^b_K \right|\norm{t(x)}\geq r^{-1/n}\right\}.
$$
Each of these should be the affinoid rigid space associated to the ring 
$$ \cur{E}_{r^{-1/n}}:= \frac{S_K\tate{T}}{(t^nT-\pi)}= \left\{\left.\sum_i a_it^i \;\right|\;  a_i\in K,\;\sup_i\norm{a_i}<\infty,\;\norm{a_i}r^{-i/n}\rightarrow 0\text{ as }i\rightarrow -\infty \right\}
$$ 
and so we find that the global sections $\Gamma(\mathbb{D}^b_K,j_{\spec{k\lser{t}}}^\dagger\cur{O}_{\mathbb{D}^b_K})$ of the ring of overconvergent functions should be nothing other than $\ekd$. 

Luckily, Huber's theory of adic spaces, or equivalently Fujiwara/Kato's theory of Zariski-Riemann spaces (as explained in \cite{huber} and \cite{rigidspaces} respectively), provides the framework in which to make these heuristics completely rigourous. Equipping $S_K$ with the $\pi$-adic topology, the adic space $\mathbb{D}^b_K:=\mathrm{Spa}(S_K,\cur{V}\pow{t})$ admits a specialisation map
$$\mathrm{sp}:\mathbb{D}^b_K \rightarrow \spec{k\pow{t}} $$
and we can define the tube of $\spec{k\lser{t}}$ inside $\mathbb{D}^b_K$, which will be the \emph{closure} of the inverse image under the specialisation map. (The reason for this is that we want these tubes to be `overconvergent' subsets of $\mathbb{D}^b_K$). A cofinal system of neighbourhoods of $]\spec{k\lser{t}}[_{\mathbb{D}^b_K}$ is then given exactly as expected, and so we do genuinely get an isomorphism
$$ \Gamma(\mathbb{D}^b_K,j_{\spec{k\lser{t}}}^\dagger\cur{O}_{\mathbb{D}^b_K})\isomto \ekd
$$ between global sections of an appropriately constructed sheaf of overconvergent functions and $\ekd$. Moreover, there will be an equivalence of categories between coherent $j^\dagger_{\spec{k\lser{t}}}\cur{O}_{\mathbb{D}^b_K}$-modules and finite dimensional $\ekd$-vector spaces. 

Thus what this is suggesting to us is that we should be looking for is a version of `relative' rigid cohomology, that is rigid cohomology relative to the frame 
$$\left(\spec{k\lser{t}},\spec{k\pow{t}},\spf{\cur{V}\pow{t}}\right)$$
rather than the frame
$$\left(\spec{k\lser{t}},\spec{k\lser{t}},\spf{\cur{O}_{\ek}}\right).$$
In other words, rather than compactifying our varieties over $k\lser{t}$ and then embedding them in smooth formal schemes over $\cur{O}_{\ek}$, we should instead compactify them over $k\pow{t}$ and then embed them in smooth, $\pi$-adic, formal schemes over $\cur{V}\pow{t}$. Thus we are lead to consider the notion of a smooth and proper frame over $\cur{V}\pow{t}$ as a triple $(X,Y,\frak{P})$ where $X$ is a $k\lser{t}$-variety, $Y$ is a proper, $k\pow{t}$-scheme with an open immersion $X\rightarrow Y$, and $\frak{P}$ is a $\pi$-adic formal $\cur{V}\pow{t}$-scheme together with a closed immersion $Y\rightarrow \frak{P}$ such that $\frak{P}$ is smooth over $\cur{V}\pow{t}$ around $X$. We can then use Huber and Fujiwara/Kato's theories to systematically work on the generic fibre of such the formal scheme $\frak{P}$.

The reader familiar with Berthelot's foundational preprint \cite{iso} might well suggest that we can make sense of the generic fibre of $\spf{\cur{V}\pow{t}}$ within Tate's original theory of rigid spaces as long as we are prepared to use the $\frak{m}=(\pi,t)$-adic, rather than the $\pi$-adic topology. There are a couple of reasons why we do not do this. First of all, if we want to compactify our varieties as schemes over $k\pow{t}$, we need the mod-$\pi$ reductions of the formal schemes of $\cur{V}\pow{t}$ we consider to be schemes, not formal schemes, over $k\pow{t}$. Another reason is that the $j^\dagger$ construction on the generic fibre of $\spf{\cur{V}\pow{t}}$ give rise to the Robba ring $\cur{R}_K$, not the bounded Robba ring $\ekd$. If we are to construct a theory over the latter, then we really must work with the $p$-adic topology on $\cur{V}\pow{t}$.

Thus we proceed to try to construct a theory of rigid cohomology relative to $$\left(\spec{k\lser{t}},\spec{k\pow{t}},\spf{\cur{V}\pow{t}}\right)$$ by using the notion of a frame $(X,Y,\frak{P})$ over $\cur{V}\pow{t}$ described above. Modulo some technical checks to ensure that the rigid spaces we want to consider are suitably nicely behaved, the theory proceeds more or less identically to `classical' rigid cohomology. We get entirely analogous `standard' systems of neighbourhoods, we have a strong fibration theorem and a Poincar\'{e} lemma, and categories of coefficients are constructed in exactly the same way (at least, `relative' coefficients are, we will return to this issue shortly). Thus a large bulk of this paper (and its sequels) consists of checking that as many of the known results about rigid cohomology as possible can also be proved in our new context. Thus our main result in this first paper is the following. 

\begin{thrm}[\ref{maintheo}, \ref{maintheocoeffs}, \ref{maindefs}] For any variety $X/k\lser{t}$ there are well-defined and functorial cohomology groups $H^i_\rig(X/\ekd)$ which are vector spaces over $\ekd$. There are also well-defined and functorial categories of coefficients $F\text{-}\mathrm{Isoc}^\dagger(X/\ekd)$, as well as well-defined and functorial cohomology groups $H^i_\rig(X/\ekd,\sh{E})$ with values in some coefficient object $\sh{E}$.
\end{thrm}

As expected, the category $F\text{-}\mathrm{Isoc}^\dagger(X/\ekd)$ will consist of certain modules with an overconvergent, integrable connection on an appropriate tube, together with a Frobenius structure. If we fix some Frobenius $\sigma$ on $\ekd$, then the cohomology groups $H^i_\rig(X/\ekd,\sh{E})$ will also come with a $\sigma$-linear endomorphism in the usual fashion.

The vast majority of the definitions, results, and proofs in the first paper will be entirely familiar to anyone well-versed in the constructions of `classical' rigid cohomology, as outlined in \cite{iso,rigcoh}, and there are almost no surprises whatsoever to be found. In parts we repeat more or less word-for-word the original proofs given by Berthelot in \cite{iso} and Le Stum in \cite{rigcoh}, and in others (whenever we have suitable base change results to hand) we can actually use the fact that these results are known over $\ek$ to give straightforward proofs. We hope that the reader will forgive us for going into such detail with material that is essentially well-known, but we considered it important to be as thorough as we thought reasonable, given the novel context.

The second paper in the series \cite{rclsf2} is then concerned with trying to prove certain basic properties that one expects of the theory $X\mapsto H^i_\rig(X/\ekd,\sh{E})$. The three most important results that one would like to know about $\ekd$-valued rigid cohomology are the following.

\begin{enumerate} \item finite dimensionality, that is $H^i_\rig(X/\ekd,\sh{E})$ should be finite dimensional over $\ekd$;
\item bijectivity of Frobenius, that is the linearised Frobenius map
$$ H^i_\rig(X/\ekd,\sh{E})\otimes_{\ekd,\sigma} \ekd \rightarrow H^i_\rig(X/\ekd,\sh{E})
$$
should be an isomorphism;
\item base change, that is there should be a functor
\begin{align*} F\text{-}\mathrm{Isoc}^\dagger(X/\ekd)&\rightarrow F\text{-}\mathrm{Isoc}^\dagger(X/\ek) \\ \sh{E}&\mapsto \hat{\sh{E}}
\end{align*}
giving rise to a natural isomorphism
$$ H^i_\rig(X/\ekd,\sh{E})\otimes_{\ekd} \ek \isomto H^i_\rig(X/\ek,\hat{\sh{E}})
$$
comparing the new theory to `classical' rigid cohomology.
\end{enumerate}
Of course, the third of these implies the other two, and in some sense is the most fundamental result. In fact, if we had a suitably robust proof of the third of these, then actually a lot of the work in these papers would be unnecessary. 

To see why, note that for any smooth and proper frame $(X,Y,\frak{P})$ over $\cur{V}\pow{t}$ in the sense outlined above, we can base change to get a frame $(X,Y_{k\lser{t}},\frak{P}_{\cur{O}_{\ek}})$, smooth and proper over $\cur{O}_{\ek}$ in the `classical' sense. If we could prove a suitable base change result for \emph{any} frame, that is that the base change of the rigid cohomology computed using $(X,Y,\frak{P})$ is the rigid cohomology computed using $(X,Y_{k\lser{t}},\frak{P}_{\cur{O}_{\ek}}),$ then finite dimensionality and invariance of the choice of frame would follow immediately from the corresponding result in `classical' rigid cohomology. However, the issue of base change for the cohomology of rigid analytic varieties is somewhat delicate, and becomes even more so when one introduces overconvegent structure sheaves, and we are not currently certain that this approach can be made to work. If it could, however, then it would given must simpler proofs of the theorems in this paper and the next, as well as proving results going beyond what we have managed so far, for example finite dimensionally in general, or cohomological descent.

Lacking such a base change theorem, we are forced to proceed in a more pedestrian manner, and can only obtain limited results. The basic idea of how to prove base change and finite dimensionality is to follow Kedlaya's proof of finite dimensionality in \cite{kedlayafiniteness}, but for reasons that we will explain shortly we were only able to make this work fully in dimension $1$. Thus the first step is to prove a suitable version of the $p$-adic local monodromy theorem, which then implies finite dimensionality and base change for $\A^1$. This is achieved again by exploiting the observation that the relation between $\ekd$ and $\ek$ is exactly analogous to the relation between a dagger algebra and its affinoid completion. Hence we can prove our required result by `descending' from the $p$-adic monodromy theorem over $\ek$, in exactly the same way as Kedlaya proves a monodromy theorem over a dagger algebra by `descending' from the completion of its fraction field. In this part of the proof, it is not really necessary to restrict to the `absolute' one-dimensional case, indeed, we see no reason why Kedlaya's methods will not apply to give a generic pushforward in relative dimension 1 exactly as in \cite{kedlayafiniteness}, which one might hope would pave the way for a general proof of finite dimensionality for smooth varieties by induction on the dimension.

It is in the second stage of the proof, however, that we find ourselves needing to make the restriction to dimension 1, and is closely linked to the difficulty in giving a `Monsky-Washnitzer' style interpretation of $\ekd$-valued rigid cohomology. The problem is in finding a suitable morphism of frames $(X,Y,\frak{P})\rightarrow (\A^n_{k\lser{t}},\P^n_{k\pow{t}},\widehat{\P}^n_{\cur{V}\pow{t}})$ extending a finite \'{e}tale morphism $X\rightarrow \A^n_{k\lser{t}}$ when $n>1$. When $n=1$ we can achieve what we need by extending the ground field so that the compactification of $X$ acquires semi-stable reduction, thus we can choose a morphism $Y\rightarrow \P^1_{k\pow{t}}$ which is a local complete intersection, and \'{e}tale away from a divisor of $\P^1_{k\lser{t}}$ - it thus lifts. In higher dimensions though, we were not able to find a suitable lifting to construct pushforward functors. Thus the most general result concerning base change that we can currently prove is the following.

\begin{thrm}[\cite{rclsf2}] Let $X/k\lser{t}$ be a smooth curve, and $\sh{E}\in F\text{-}\mathrm{Isoc}^\dagger(X/\ekd)$. Then the base change map
$$ H^i_\rig(X/\ekd,\sh{E})\otimes_{\ekd} \ek \rightarrow H^i_\rig(X/\ek,\hat{\sh{E}})
$$
is an isomorphism.
\end{thrm}

We also introduce $\ekd$-valued rigid cohomology with compact supports, and can prove a similar base change theorem for smooth curves, although with restrictions on the coefficients. Thus by using the fact that Poincar\'{e} duality is known over $\ek$, we easily obtain the following result.

\begin{thrm}[\cite{rclsf2}] For any smooth curve over $k\lser{t}$ there is a trace map
$$ \mathrm{Tr}: H^2_{c,\rig}(X/\ekd)\rightarrow \ekd(-1)
$$ such that for any $\sh{E}\in F\text{-}\mathrm{Isoc}^\dagger(X/\ekd)$ which extends to $\overline{\sh{E}}\in F\text{-}\mathrm{Isoc}^\dagger(\overline{X}/\ekd)$ on a smooth compactification of $X$, the induced pairing
$$ H^i_{\rig}(X/\ekd,\sh{E}) \times H^{2-i}_{c,\rig}(X/\ekd,\sh{E}^\vee)\rightarrow \ekd(-1)
$$
is a perfect pairing of $\varphi$-modules over $\ekd$.
\end{thrm}

Finally, in the third paper \cite{rclsf3} in this series, we get to some arithmetic applications. The first task is to introduce a more refined category of coefficients, one for which the differential structure is relative to $K$, rather than $\ekd$. Objects $\sh{E}$ in the resulting category $\mathrm{Isoc}^\dagger(X/K)$ will then have connections on their cohomology groups $H^i_\rig(X/\ekd,\sh{E})$, arising via the Gauss--Manin construction. Thus if $\sh{E}$ is equipped with a Frobenius structure, i.e. is an object of the category $F\text{-}\mathrm{Isoc}^\dagger(X/K)$, and $X$ is a smooth curve over $k\lser{t}$, then the cohomology groups $H^i_\rig(X/\ekd,\sh{E})$ will be $(\varphi,\nabla)$-modules over $\ekd$. Actually, to prove that the $\sigma$ and $\nabla$-module structures are compatible is slightly delicate, and necessitates a discussion of descent, this also has the added bonus of extending $\ekd$-valued rigid cohomology to non-embeddable varieties. We also get connection on cohomology groups with compact support, and the Poincar\'{e} pairing will turn out to be a perfect pairing of $(\varphi,\nabla)$-modules over $\ekd$. Using this theory, we can then attach $p$-adic Weil--Deligne representations to smooth curves over $k\lser{t}$, the point being that $H^i_\rig(X/\cur{R}_K):=H^i_\rig(X/\ekd)\otimes_{\ekd}\cur{R}_K$ will be a $(\varphi,\nabla)$-module over $\cur{R}_K$, and hence we can use Mamora's procedure from \cite{marmora} to produce an associated Weil--Deligne representation. This can then be compared with the $\ell$-adic Weil--Deligne representations coming from $\ell$-adic cohomology. For more details see the introduction to \cite{rclsf3}. 

\section{Rigid cohomology and adic spaces}\label{opening}

In this section $k$ is again a field of characteristic $p$, $\cur{V}$ a complete DVR whose residue field is $k$ and fraction field $K$ is of characteristic $0$. Moreover $\pi$ is a uniformiser for $\cur{V}$, and as above, we fix a norm $\norm{\cdot}$ on $K$ such that
$\norm{p}=p^{-1}$, and let $r=\norm{\pi}^{-1}>1$. Berthelot's theory of rigid cohomology
$$
X\mapsto H^*_\rig(X/K)
$$
is a $p$-adic cohomology theory for $k$-varieties, whose construction we quickly recall. To define $H^*_\rig(X/K)$ one first compactifies $X$ into a proper scheme $Y/k$, and them embeds $Y$ into a formal scheme $\mathfrak{P}/\cur{V}$, which is smooth over $\cur{V}$ in a neighbourhood of $X$. One then considers the generic fibre $\mathfrak{P}_K$ of $\mathfrak{P}$, which is a rigid analytic space, and one has a specialisation map
$$ \mathrm{sp}:\frak{P}_K\rightarrow P 
$$ where $P$ is the special fibre of $\frak{P}$, that is its mod-$\pi$ reduction. Assocaited to the subschemes $X,Y\subset P$ one has the tubes
$$ ]X[_\frak{P}:=\mathrm{sp}^{-1}(X),\;\;]Y[_\frak{P}=\mathrm{sp}^{-1}(Y),
$$
let $j:]X[_\frak{P}\rightarrow ]Y[_\frak{P}$ denote the inclusion. One then takes $j_X^\dagger\Omega^*_{]Y[_\mathfrak{P}}$ to be the subsheaf of $j_*\Omega^*_{]X[_\frak{P}}$ consisting of \emph{overconvergent} differential forms, that is those that converge on some strict neighbourhood of $]X[_\frak{P}$ inside $]Y[_\frak{P}$. The rigid cohomology of $X$ is then 
$$ H^i_\rig(X/K):= H^i(]Y[_\frak{P},j_X^\dagger\Omega^*_{]Y[_\frak{P}}),
$$
this is independent of both $Y$ and $\mathfrak{P}$.

In the theory that we wish to construct, we will want to consider the `generic fibres' of more general formal schemes, namely $\pi$-adic formal schemes topologically of finite type over $\cur{V}\pow{t}$, and as such this falls somewhat outside the scope of Tate's theory of rigid spaces. Luckily, this is nicely covered by Huber's theory of adic spaces, or equivalently, Fujiwara-Kato's theory of Zariski-Riemann spaces (the equivalence of these two perspectives, at least in all the cases we will need in this article, is Theorem II.A.5.2 of \cite{rigidspaces}). Thus as a warm-up for the rest of the paper, as well as to ensure `compatibility' of our new theory with traditional rigid cohomology, in this opening section we show that rigid cohomology can be computed using adic spaces. 

This is rather straightforward, and is achieved more or less by showing that certain cofinal systems of neighbourhoods for $]X[_\mathfrak{P}$ in $]Y[_\mathfrak{P}$ inside the rigid space $\mathfrak{P}_K$ are also cofinal systems of strict neighbourhoods of $]X[_\mathfrak{P}$ in $]Y[_\mathfrak{P}$ inside the corresponding adic space. We can then use the fact that corresponding rigid and adic spaces have the same underlying topoi to conclude that the two different constructions of $j_X^\dagger\Omega^*_{]Y[_\mathfrak{P}}$ (in the rigid and adic worlds) give \emph{the same} object in the appropriate topos, and hence have the same cohomology. 

So let $(X,Y,\mathfrak{P})$ be a smooth and proper frame, as appearing in Berthelot's construction, that is $X\rightarrow Y$ is an open immersion of $k$-varieties, $Y$ is proper over $k$, $Y\rightarrow \mathfrak{P}$ is a closed immersion of formal $\cur{V}$-schemes and $\mathfrak{P}$ flat over $\cur{V}$ and formally smooth over $\cur{V}$ in some neighbourhood of $X$. Let $P$ denote the special fibre of $\mathfrak{P}$, so that there is a homeomorphism of topological spaces $P \simeq \mathfrak{P}$, and let $Z=Y\setminus X$, with some closed subscheme structure.

\begin{note} A \emph{variety} will always mean a separated scheme of finite type, and formal schemes over $\cur{V}$ will always be assumed to be separated, $\pi$-adic and topologically of finite type.
\end{note}

In this situation, we will want to consider three different sorts of generic fibre of $\mathfrak{P}$, rigid, Berkovich, and adic. To describe them, we work locally on $\mathfrak{P}$, and assume it to be of the form $\mathrm{Spf}(A)$ for some topologically finite type $\cur{V}$-algebra $A$, for any such $A$, we we let $A^+$ denote the integral closure of $A$ inside $A_K:=A\otimes_{\cur{V}}K$.

\begin{itemize} \item The rigid generic fibre $\mathfrak{P}^\rig$. This is the set $\mathrm{Sp}(A_K)$ of maximal ideals of $A_K$, considered as a locally $G$-ringed space in the usual way (see for example Chapter 4 of \cite{freput}). Alternatively, this is the collection of (equivalence classes of) \emph{discrete} continuous valuations $v:A_K\rightarrow \left\{0 \right\}\cup \R^{>0}$.
\item The Berkovich generic fibre $\mathfrak{P}^\ber$. This is the set $\sh{M}(A_K)$ of (equivalence classes of) continuous rank $1$ valuations $v:A_K\rightarrow \left\{0\right\}\cup \R^{>0}$, considered as a topological space as in Chapter 1 of \cite{berk2}.
\item The adic generic fibre. This is the set $\mathrm{Spa}(A_K,A^+)$ of (equivalence classes of) continuous valuations $v:A_K\rightarrow \left\{0\right\} \cup \Gamma$ into some totally ordered group $\Gamma$ (of possibly rank $>1$), satisfying $v(A^+)\leq 1$. It is considered as a locally ringed space as in \cite{huber}.
\end{itemize}

\begin{rem} Whenever $B$ is a topologically finite type $K$-algebra, that is a quotient of some Tate algebra $K\tate{x_1,\ldots,x_n}$, we will also write $B^+$ for the integral closure of the image of $\cur{V}\tate{x_1,\ldots,x_n}$ inside $B$, and $\spa{B}$ instead of $\spa{B,B^+}$. Note that $B^+$ does not depend on the choice of presentation of $B$.
\end{rem}

\begin{rem} It is generally conventional when working with higher rank valuations for them to be written multiplicatively, and we will do so throughout this article. Hence the slightly strange looking definition of the Berkovich space $\sh{M}(A_K)$. 
\end{rem}

There are several relations among these spaces, for example, there is an obvious inclusion $\mathfrak{P}^\ber\rightarrow \mathfrak{P}^\ad$ which is \emph{not} continuous, but there \emph{is} a continuous map $[\cdot]:\mathfrak{P}^\ad\rightarrow \mathfrak{P}^\ber$ which exhibits $\mathfrak{P}^\ber$ as the maximal separated (Hausdorff) quotient of $\mathfrak{P}^\ad$ (as follows from Proposition II.C.1.8 of \cite{rigidspaces}). There is also an inclusion $\mathfrak{P}^\rig\rightarrow \mathfrak{P}^\ad$ as the subset of \emph{rigid} points, this factors throughout $\frak{P}^\ber$. For $x$ a point of any of these spaces, we will write $v_x(\cdot)$ for the corresponding valuation, note this is compatible with the embeddings $\mathfrak{P}^\rig
\rightarrow \mathfrak{P}^\ber\rightarrow \mathfrak{P}^\ad$ but \emph{not} with the map $[\cdot]:\mathfrak{P}^\ad\rightarrow \mathfrak{P}^\ber$. 

For $\#\in\left\{\rig,\ber,\ad\right\}$ there are specialisation maps
$$
\mathrm{sp}:\mathfrak{P}^\# \rightarrow \mathfrak{P} \simeq P
$$
which are compatible with the inclusions $\mathfrak{P}^\rig\rightarrow \mathfrak{P}^\ber\rightarrow \mathfrak{P}^\ad$, but \emph{not} with the quotient map $[\cdot]:\mathfrak{P}^\ad\rightarrow \mathfrak{P}^\ber$. The maps $\mathfrak{P}^\rig\rightarrow \mathfrak{P}$ and $\mathfrak{P}^\ad\rightarrow \mathfrak{P}$ are continuous (for the $G$-topology on $\mathfrak{P}^\rig$), but the map $\mathfrak{P}^\ber\rightarrow\mathfrak{P}$ is anti-continuous, that is the inverse image of an open set is closed and vice versa.

\begin{defn} For the closed subvariety $Y\subset P$ define the tubes
\begin{align*} ]Y[_\mathfrak{P}^\rig&:=\mathrm{sp}^{-1}(Y)\subset \mathfrak{P}^\rig \\
]Y[_\mathfrak{P}^\ber&:=\mathrm{sp}^{-1}(Y)\subset \mathfrak{P}^\ber \\
]Y[_\mathfrak{P}^\ad&:=\mathrm{sp}^{-1}(Y)^\circ\subset \mathfrak{P}^\ad,
\end{align*}
note the fact that the adic tube is the interior of the `na\"{i}ve' tube $\mathrm{sp}^{-1}(Y)$. Also note that this definition works for any closed subset of $P$, in particular we can talk about the tubes $]Z[^\#_\mathfrak{P}$.
\end{defn}

These tubes can be calculated locally as follows. Suppose that $\mathfrak{P}=\spf{A}$ is affine, and that $f_1,\ldots,f_n\in A$ are functions such that $Y\subset P$ is the vanishing locus of the reductions $\bar{f}_i$. Then we have:

\begin{align*}
]Y[_\mathfrak{P}^\rig&=\left\{\left.x\in\mathfrak{P}^\rig \;\right|\; v_x(f_i) <1 \;\forall i\right\} \\
]Y[_\mathfrak{P}^\ber&=\left\{\left.x\in\mathfrak{P}^\ber \;\right|\; v_x(f_i) <1 \;\forall i\right\} \\
]Y[_\mathfrak{P}^\ad&=\left\{\left.x\in\mathfrak{P}^\ad \;\right|\; v_{[x]}(f_i) <1 \;\forall i\right\} 
\end{align*}
again note the difference in the description of the adic tube. Also worth noting is the fact that $]Y[_\mathfrak{P}^\ad$ is the inverse image of $]Y[_\mathfrak{P}^\ber$ under the map $[\cdot]:\mathfrak{P}^\ad\rightarrow \mathfrak{P}^\ber$.

\begin{lem} In the above situation we have:
\begin{align*}]Y[_\mathfrak{P}^\ad &=\bigcup_{n\geq1} [Y]_n^\ad\textrm{, where}  \\
[Y]_n^\ad&= \left\{\left. x\in\mathfrak{P}^\ad\;\right|\; v_x(\pi^{-1}f_i^n)\leq1  \;\forall i\right\}, 
\end{align*} and
\begin{align*}
 ]Y[_\mathfrak{P}^\rig &=\bigcup_{n\geq1} [Y]_n^\rig\textrm{, where}  \\
 [Y]_n^\rig &= \left\{ \left. x\in\mathfrak{P}^\rig\;\right|\; v_x(\pi^{-1}f_i^n)\leq1  \;\forall i\right\}
\end{align*}
are (admissible) affinoid coverings.
\end{lem}

\begin{proof} Note that each $[Y]_n^\rig=\mathrm{Sp}(A_K\tate{T}/(\pi T-f^n))=\left\{\left. x\in\mathfrak{P}^\rig \;\right|\; v_x(f)\leq r^{-1/n} \right\}$ is affinoid, and these form an admissible cover of $]Y[^\rig_\mathfrak{P}$ by Proposition 1.1.9 of \cite{iso}.

In the adic case, note that each $[Y]_n^\ad=\mathrm{Spa}(A_K\tate{T}/(\pi T-f^n) )$ is an affinoid open subspace of $\mathfrak{P}^\ad$, we must show that they form an open cover of $]Y[_\mathfrak{P}^\ad$. Firstly, note that $v_x(\pi^{-1}f^n)\leq 1\Rightarrow v_{[x]}(\pi^{-1}f^n)\leq 1\Rightarrow v_{[x]}(f)\leq r^{-1/n}<1$ (recall that $r=\norm{\pi}^{-1}$), the first implication following from Lemma \ref{vals} below and the second using the fact that $v_{[x]}$ is a valuation of rank $1$, so can be viewed as a multiplicative map into $\R^{\geq0}$. Hence $[Y]_n^\ad\subset ]Y[_\mathfrak{P}^\ad$ for each $n$.

Now suppose that $x\in ]Y[_\mathfrak{P}^\ad$, that is $v_{[x]}(f_i)<1$. Then there exists some $n$ such that $v_{[x]}(\pi^{-1}f_i^n)<1$. Hence again by Lemma \ref{vals} below we must have $v_{x}(\pi^{-1}f_i^n)\leq 1$ and hence $x\in [Y]^\ad_n$ for some $n$.
\end{proof}

\begin{lem} \label{vals} Let $\sh{X}=\spa{B}$ be an affinoid adic space for some $B$ topologically of finite type over $K$. Then for any point $x\in \sh{X}$ and any $f\in B$ we have
\begin{align*} v_x(f)\leq1&\Rightarrow v_{[x]}(f)\leq1 \\ 
v_{[x]}(f) <1 &\Rightarrow v_{x}(f)< 1.
\end{align*}
\end{lem}

\begin{proof} Let $I\subset B$ denote the support of the valuation $v_x$ corresponding to $x$, that is the ideal of elements with valuation $0$. Let $V_x$ be the valuation ring of the induced valuation $v:\mathrm{Frac}(B/I) \rightarrow \left\{ 0\right\} \cup\Gamma$, and let $P_x\subset V_x$ denote the prime ideal of elements whose valuation is $<1$. The radical $\frak{p}:=\sqrt{(\pi)}\subset P_x$ of $(\pi)$ is a height one prime ideal of $V_x$ to which we may associate a rank one valuation $v_\mathfrak{p}:\mathrm{Frac}(B/I)\rightarrow \left\{0\right\}\cup \Gamma'$, which is the valuation associated to $[x]$ (see for example II.3.3.(b) of \cite{rigidspaces}). Since $\mathfrak{p}\subset P_x$ it follows that $v(\lambda)\leq 1\Rightarrow v_\mathfrak{p}(\lambda)\leq1$ for all $\lambda\in B/I$, which proves the first claim. For the second, note that we may assume that $v_x(f)\neq0$, that is $f\notin I$, and hence $v_x$ and $v_{[x]}$ both extend uniquely to valuations on $B\tate{f^{-1}}$. To obtain the second claim we now just simply apply the first to $f^{-1}$. \end{proof}

In particular, $]Y[_\mathfrak{P}^\ad$ is an adic space locally of finite type over $\spa{K}$. In II.B of \cite{rigidspaces}, Fujiwara and Kato construct an equivalence
$$
\sh{X}\mapsto \sh{X}_0
$$
from the category of adic spaces locally of finite type over $\spa{K}$ to rigid spaces locally of finite type over $\mathrm{Sp}(K)$, which is such that $\spa{B}_0=\mathrm{Sp}(B)$ for any affinoid algebra $B$, and such that $(\mathfrak{P}^\ad)_0= \mathfrak{P}^\rig$ for any formal scheme $\mathfrak{P}$ of the type considered above. The previous lemma allows us to deduce the same result for the tube $]Y[_\mathfrak{P}$.

\begin{cor} \label{tubes}There is an isomorphism
$$ (]Y[^\ad_\mathfrak{P})_0\cong ]Y[_\mathfrak{P}^\rig $$
as rigid spaces over $\mathrm{Sp}(K)$.
\end{cor}

This should convince any doubtful reader that the definition above of $]Y[_\mathfrak{P}^\ad$ is the correct one. Note that in all cases $\#\in\left\{\rig,\ber,\ad\right\}$, the specialisation map gives rise to a map
$$
]Y[_\mathfrak{P}^\#\rightarrow Y.
$$

\begin{defn} For the open subvariety $X\subset Y$ define the tubes
\begin{align*} ]X[_\mathfrak{P}^\rig&:=\mathrm{sp}^{-1}(X)\subset ]Y[_\mathfrak{P}^\rig \\
]X[_\mathfrak{P}^\ber&:=\mathrm{sp}^{-1}(X)\subset ]Y[_\mathfrak{P}^\ber \\
]X[_\mathfrak{P}^\ad&:=\overline{\mathrm{sp}^{-1}(X)}\subset ]Y[_\mathfrak{P}^\ad,
\end{align*}
again note the fact that the adic tube is the closure of the `na\"{i}ve' tube $\mathrm{sp}^{-1}(X)$. 
\end{defn}

As before, these tubes can be calculated locally as follows. Suppose that $\mathfrak{P}=\spf{A}$ is affine, and suppose that $g_1,g_2,\ldots,g_m\in A$ are functions such that $X=Y\cap ( \cup_j D(\bar{g}_j) ) $, where $\bar{g}_j$ the reduction of $g_j$. Then we have:

\begin{align*}
]X[_\mathfrak{P}^\rig&=\left\{\left.x\in]Y[_\mathfrak{P}^\rig \;\right|\; \exists j \text{ s.t. }v_x(g_j) \geq1 \right\} \\
]X[_\mathfrak{P}^\ber&=\left\{\left.x\in]Y[_\mathfrak{P}^\ber \;\right|\; \exists j \text{ s.t. }v_x(g_j) \geq1 \right\} \\
]X[_\mathfrak{P}^\ad&=\left\{\left.x\in]Y[_\mathfrak{P}^\ad \;\right|\; \exists j \text{ s.t. }v_{[x]}(g_j) \geq1 \right\} 
\end{align*}
and again note that $]X[_\mathfrak{P}^\ad$ is the inverse image of $]X[_\mathfrak{P}^\ber$ under the map $[\cdot]:\mathfrak{P}^\ad\rightarrow \mathfrak{P}^\ber$. For $\#\in\left\{\rig,\ber,\ad\right\}$ denote by
$$
j:]X[_\mathfrak{P}^\#\rightarrow ]Y[_\mathfrak{P}^\#
$$
the canonical inclusion, note that for $\#=\ber,\ad$ this is the inclusion of a \emph{closed} subset, but for $\#=\rig$ this is an \emph{open} immersion. For $\#=\ber,\ad$ and a sheaf $\sh{F}$ on $]Y[_\mathfrak{P}^\#$ we define $j_X^\dagger\sh{F}:=j_*j^{-1}\sh{F}$, however, the definition in the rigid case is slightly more involved. 

\begin{defn} A strict neighbourhood of $]X[_\mathfrak{P}^\rig$ in $]Y[_\mathfrak{P}^\rig$ is an open subset $V\subset ]Y[_\mathfrak{P}^\rig$ such that $]Y[_\mathfrak{P}^\rig= V\cup ]Z[_\mathfrak{P}^\rig$ is an admissible open covering, where recall that $Z$ is the complement of $X$ in $Y$. For any such $V$, we let $j_V:V\rightarrow ]Y[_\mathfrak{P}^\rig$ be the canonical open immersion.
\end{defn}

For a sheaf $\sh{F}$ on $]Y[_\mathfrak{P}^\rig$, define $j_X^\dagger\sh{F}=\mathrm{colim}_V j_{V*}j_V^{-1}\sh{F}$, where the colimit is taken over all strict neighbourhoods $V$ of $]X[_\mathfrak{P}^\rig$ in $]Y[_\mathfrak{P}^\rig$. 

\begin{defn} The rigid cohomology of $X$ is defined to be 
$$H^i_\rig(X/K):= H^i(]Y[_\mathfrak{P}^\rig,j_X^\dagger\Omega^*_{]Y[_\mathfrak{P}^\rig}). $$ This does not depend on the choice of frame $(X,Y,\mathfrak{P})$.
\end{defn}

We want to show that we can compute this instead as
$$
H^i(]Y[_\mathfrak{P}^\ad,j_X^\dagger\Omega^*_{]Y[_\mathfrak{P}^\ad})=H^i(]X[_\mathfrak{P}^\ad,j^{-1}\Omega^*_{]Y[_\mathfrak{P}^\ad}).
$$
In order to do this we must first recall Berthelot's construction of a cofinal system of strict neighbourhoods from Section 1.2 of \cite{iso}. For $\mathfrak{P}$ affine we have constructed affinoids $[Y]_n^\rig$ and $[Y]_n^\ad$, which depended on the choice of functions $f_i\in\cur{O}_\mathfrak{P}$ cutting out $Y$ in the reduction $P$.

\begin{lem} For $n\gg0$ the affinoids $[Y]_n^\rig$ and $[Y]_n^\ad$ are independent of the choice of the $f_i$. Hence they glue over an open affine covering of $\mathfrak{P}$.
\end{lem}

\begin{proof} For $[Y]_n^\rig$ this is proved in 1.1.8 of \cite{iso}, and the proof for $[Y]_n^\ad$ is identical.
\end{proof}

Note that $([Y]^\ad_n)_0\cong [Y]^\rig_n$, and that these are the same as the closed tubes $[Y]_{r^{-1/n}}$ of radius $r^{-1/n}$, constructed by Berthelot in \cite{iso}. We also have
$$
]Y[^\#_\mathfrak{P}=\bigcup_{n\gg0} [Y]^\#_{n}
$$
and this is an admissible covering if $\#=\rig$. Similarly, when $\mathfrak{P}$ is affine and we have $f_i,g_j$ as above, so that $Y=\cap_i Z(\bar{f}_i)$ and $X=Y\cap (\cup_j D(\bar{g}_j))$ we can define
\begin{align*}
U_m^\rig &=\left\{\left.x\in]Y[_\mathfrak{P}^\rig \;\right|\; \exists j \text{ s.t. }v_x(\pi^{-1}g_j^m) \geq1 \right\} \\
U_m^\ad &= \left\{\left.x\in]Y[_\mathfrak{P}^\rig \;\right|\; \exists j \text{ s.t. }v_x(\pi^{-1}g_j^m) \geq1 \right\}
\end{align*}
as well as 
\begin{align*}
U_{m,j}^\rig &=\left\{\left.x\in]Y[_\mathfrak{P}^\rig \;\right|\; v_x(\pi^{-1}g_j^m) \geq1 \right\} \\
U_{m,j}^\ad &= \left\{\left.x\in]Y[_\mathfrak{P}^\rig \;\right|\; v_x(\pi^{-1}g_j^m) \geq1 \right\}
\end{align*}
so that $U_m^\#=\cup_jU_{m,j}^\#$, and this is an admissible open covering when $\#=\rig$. As before, for $m\gg 0$ these are independent of the choice of the $g_j$ and hence glue over an open affine covering of $\mathfrak{P}$. Finally we set
\begin{align*}
V_{n,m}^\#&=[Y]^\#_{n}\cap U_{m}^\# \\
V_{n,m,j}^\#&=[Y]^\#_{n}\cap U_{m,j}^\#
\end{align*}
so that $V_{n,m}^\#=\cup_jV_{n,m,j}^\#$ and this covering is admissible when $\#=\rig$. 

\begin{lem} When $\mathfrak{P}$ is affine, the $V_{n,m,j}^\#$ are affinoid, and $(V_{n,m,j}^\ad)_0\cong V_{n,m,j}^\rig$. 
\end{lem}

\begin{proof} Suppose that $\mathfrak{P}\cong\spf{A}$ is affine, and choose $f_i,g_j$ as above. Then
\begin{align*} V_{n,m,j}^\ad&\cong \spa{\frac{A_K\tate{T_1,\ldots,T_n,S}}{(\pi T_i-f_i,\pi-g_j^mS)}} \\
 V_{n,m,j}^\rig&\cong \mathrm{Sp}\left(\frac{A_K\tate{T_1,\ldots,T_n,S}}{(\pi T_i-f_i,\pi-g_j^mS)}\right)
\end{align*}
and the lemma follows.
\end{proof}

\begin{cor} For all frames $(X,Y,\mathfrak{P})$, we have $(V_{n,m}^\ad)_0\cong V_{n,m}^\rig$.\qed
\end{cor}

Now, for any increasing sequence of integers $\underline{m}(n)\rightarrow \infty$, we let 
$$
V^\#_{\underline{m}}=\bigcup_n V^\#_{n,\underline{m}(n)}.
$$
The previous corollary tells us that $(V^\ad_{\underline{m}})_0\cong V^\rig_{\underline{m}}$, and it it proved in 1.2.4 of \cite{iso} that the $V^\rig_{\underline{m}}$ for varying $\underline{m}$ form a cofinal system of strict neighbourhoods of $]X[^\rig_\mathfrak{P}$ inside $]Y[^\rig_\mathfrak{P}$. In order to show that the same is true in the adic world, we need the following lemma.

\begin{lem} \label{cofinalkey} Let $\sh{X}=\mathrm{Spa}(B)$ be an affinoid rigid space, locally of finite type over $K$. Let $V\subset \sh{X}$ be an open subset, and $g\in B$ such that
$$
V\supset \left\{\left.x\in\sh{X} \;\right|\; v_{[x]}(g)\geq1\right\}.
$$
Then there exists some $m$ such that 
$$
V\supset \left\{\left.x\in\sh{X} \;\right|\; v_x(\pi^{-1}g^m) \geq1 \right\}.
$$
\end{lem}

\begin{proof} Let $T=\sh{X}\setminus V$ denote the complement of $V$, this is a quasi-compact topological space. As in \S II.4.3 of \cite{rigidspaces}, $g$ defines a continuous function $\Norm{g}:\sh{X}\rightarrow \R^{\geq0}$ (although the function there depends upon a choice of an ideal of definition $\sh{I}$ and a non-zero constant $c<1$, there are canonical choices in our case, namely $\sh{I}=(\pi)$ and $c$ such that the induced norm on constant functions is the fixed norm on $K$).

This induces a continuous function 
$$ \Norm{g}:T\rightarrow \R^{\geq0}
$$ 
which is in effect a consistent normalisation of $\norm{v_{[x]}(g)}\in\R$ for varying $x$. Therefore by assumption $\Norm{g}(T)\subset [0,1)$. But since $T$ is quasi-compact, so must its image under $\Norm{g}$ be, and hence $\Norm{g}(T)\subset [0,\eta]$ for some $\eta<1$. Hence
$$
T\subset \left\{\left.x\in \sh{X} \;\right|\; v_{[x]}(g) \leq \eta \right\}
$$
and by Lemma \ref{vals}, there exists some $m$ such that 
$$ T \subset \left\{\left. x\in \sh{X} \;\right|\; v_x(\pi^{-1}g^m) \leq 1\right\}.
$$
The claim follows.
\end{proof}

\begin{prop} \label{cofinal} As $\underline{m}$ varies, the $V^\ad_{\underline{m}}$ form a cofinal system of open neighbourhoods of $]X[^\ad_\mathfrak{P}$ inside $]Y[^\ad_\mathfrak{P}$.
\end{prop}

\begin{proof} Let $V\subset ]Y[_\mathfrak{P}$ be an open subset containing $]X[_\mathfrak{P}$. It suffices to show that for all $n$ there exists some $m$ such that
$$
 [Y]_n \cap V \supset [Y]_n \cap U_m.
$$
Since the $[Y]_n$ are quasi-compact and glue over an open affine covering of $\mathfrak{P}$, we may assume that $\mathfrak{P}$ is affine, and hence the $[Y]_n$ are affinoid. Let $g_j\in\cur{O}_\mathfrak{P}$ be functions whose reductions $\bar{g}_j$ satisfy $X= Y\cap (\cup_j D(\bar{g}_j))$, so that
\begin{align*}
[Y]_n \cap U_m &= \cup_j U_{m,j} \\
[Y]_n \cap U_{m,j} &= \left\{ \left. x\in [Y]_n \;\right|\; v_x(\pi^{-1}g_j^m)\geq1 \right\}.
\end{align*}
It thus suffices to show that for all $j$, there exists $m$ such that $V\cap [Y]_n \supset [Y]_n \cap U_{m,j}$. But this is exactly the content of Lemma \ref{cofinalkey} above.
\end{proof}

Before we prove the fundamental result of this section, Proposition \ref{comp1}, we need the following topological lemma.

\begin{lem} \label{closed} Let $i: T\rightarrow V$ be the inclusion of a closed subspace $T$ of a topological space $V$. Suppose that there exists a basis $\cur{B}$ of open subsets of $V$ such that for every $W\in\cur{B}$ and every open subset $U'$ of $V$ containing $T\cap W$, there exists an open neighbourhood $U$ of $T$ in $V$ such that $U\cap W\subset U'$. Then for any sheaf $\sh{F}$ on $V$ there exists an isomorphism
$$
i_*i^{-1}\sh{F}\cong \mathrm{colim}_{U\supset T} j_{U*}j_U^{-1}\sh{F}
$$
where the colimit runs over all open neighbourhoods $U$ of $T$ in $V$, and $j_{U}:U\rightarrow V$ denotes the corresponding inclusion. 
\end{lem}

\begin{proof} Note that by general nonsense, $i^{-1}$ commutes with sheafification, we claim that the same is actually true for $i_*$. Indeed, for any presheaf $\sh{G}$ there is a natural morphism
$$
(i_*\sh{G})^a\rightarrow i_*(\sh{G}^a)
$$
where $(-)^a$ denotes sheafification. To check that it is an isomorphism, we can check on stalks. For any point $x\notin T$, the stalks of both sides at $x$ are $0$, and for any point $x\in T$, the stalks of both sides at $x$ is just the stalk $\sh{G}_x$.

It thus follows that $i_*i^{-1}\sh{F}$ is the sheafification of the presheaf
$$
W\mapsto \mathrm{colim}_{U'\supset T\cap W}\Gamma(U',\sh{F}).
$$
Since sheafification preserves colimits, it follows that $\mathrm{colim}_{U\supset T} j_{U*}j_U^{-1}\sh{F}$ is the sheafification of the presheaf
$$
W\mapsto \mathrm{colim}_{U\supset T}\Gamma(U\cap W,\sh{F}).
$$
thus there is a natural map
$$
\mathrm{colim}_{U\supset T} j_{U*}j_U^{-1}\sh{F}\rightarrow i_*i^{-1}\sh{F}
$$
which is induced by
\begin{align*} \left\{ U\supset T \right\} &\rightarrow \left\{ U'\supset W\cap T \right\} \\
U&\mapsto  U\cap W.
\end{align*}
The condition in the statement of the lemma is exactly that this is a cofinal map of directed sets for a basis for the topology of $V$.
\end{proof}

\begin{prop} \label{comp1} Under the equivalence $(-)_0:(]Y[_\mathfrak{P}^\ad,\cur{O}_{]Y[_\mathfrak{P}^\ad})\cong(]Y[_\mathfrak{P}^\rig,\cur{O}_{]Y[_\mathfrak{P}^\rig}) $ of ringed topoi induced by Corollary \ref{tubes} and II.B.2(e) of \cite{rigidspaces}, we have an isomorphism
$$
(j_X^\dagger \sh{F})_0\cong j_X^\dagger(\sh{F}_0)
$$
for any $\cur{O}_{]Y[_\mathfrak{P}^\ad}$-module $\sh{F}$.
\end{prop}

\begin{proof} First note that Proposition \ref{cofinal} also holds when we restrict to an affinoid subset of $\mathfrak{P}^\ad$, and hence the conditions of Lemma \ref{closed} are met for the inclusion $]X[_\mathfrak{P}^\ad\rightarrow ]Y[_\mathfrak{P}^\ad$, and we have
$$
j_X^\dagger\sh{F}\cong \mathrm{colim}_{\underline{m}} j_{\underline{m}*}j_{\underline{m}}^{-1}\sh{F}
$$
where $j_{\underline{m}}:V^\ad_{\underline{m}}\rightarrow ]Y[^\ad_\mathfrak{P}$ denotes the inclusion. The functor $(-)_0$ commutes with push-forward and pullback, and hence by 1.2.4 of \cite{iso}, which proves an analogue of Proposition \ref{cofinal} in the rigid world, we have
\begin{align*}
(j_X^\dagger\sh{F})_0 &\cong (\mathrm{colim}_{\underline{m}} j_{\underline{m}*}j_{\underline{m}}^{-1}\sh{F})_0 \\
&\cong \mathrm{colim}_{\underline{m}} j_{\underline{m}0*}j_{\underline{m}0}^{-1}\sh{F}_0 \\
&\cong j_X^\dagger(\sh{F}_0) 
\end{align*}
as required.
\end{proof}

\begin{cor} There is an isomorphism
$$
H^i(]Y[^\ad_\mathfrak{P},j_X^\dagger\Omega^*_{]Y[^\ad_\mathfrak{P}/K})\cong H^i(]Y[^\rig_\mathfrak{P},j_X^\dagger\Omega^*_{]Y[^\rig_\mathfrak{P}/K}).
$$ 
\end{cor}

\begin{proof} This follows from the previous corollary together with the fact that there is an isomorphism $(\Omega^*_{]Y[^\ad_\mathfrak{P}/K})_0\cong \Omega^*_{]Y[^\rig_\mathfrak{P}/K}$.
\end{proof}

\section{Rigid cohomology over Laurent series fields}\label{rcolsf}

Let $k,\cur{V},K,\pi,r$ be as in the previous section. As discussed in the introduction, if we take our ground field to be $k\lser{t}$, the Laurent series field over $k$, then rigid cohomology is a functor
$$
X\mapsto H^*_\rig(X/\cur{E}_K)
$$
taking values in vector spaces over the Amice ring
$$
\cur{E}_K=\left\{\left.\sum_i a_it^i \in K\pow{t,t^{-1}}\;\right|\;  \sup_i\norm{a_i}<\infty,\;a_i\rightarrow 0\text{ as }i\rightarrow -\infty \right\}.
$$
Again, as we said there, if we are to obtain a theory 
$$
X\mapsto H^*_\rig(X/\cur{E}^\dagger_K)
$$
taking values in the bounded Robba ring
$$
\cur{E}_K^\dagger = \left\{\left.\sum_i a_it^i \in\cur{E}_K\;\right|\;  \exists \eta<1 \text{ s.t. } \norm{a_i}\eta^i\rightarrow 0\text{ as }i\rightarrow -\infty \right\}
$$
then we need to take into account overconvergence conditions along $t=0$, or, in other words, we should compactify our varieties over $k\pow{t}$ rather than over $k\lser{t}$. This leads to the following definition.

\begin{defn}\label{framevt} A frame over $\cur{V}\pow{t}$ is a triple $(X,Y,\mathfrak{P})$ where $X\rightarrow Y$ is an open immersion of a $k\lser{t}$-variety $X$ into a separated, $k\pow{t}$-scheme $Y$ of finite type, and $Y\rightarrow \mathfrak{P}$ is a closed immersion of $Y$ into a separated, topologically finite type, $\pi$-adic formal $\cur{V}\pow{t}$-scheme. We say that a frame is proper if $Y$ is proper over $k\pow{t}$ and smooth if $\mathfrak{P}$ is smooth over $\cur{V}\pow{t}$ in a neighbourhood of $X$. We say that a $k\lser{t}$-variety $X$ is embeddable if there exists a smooth and proper frame of the form $(X,Y,\mathfrak{P})$.
\end{defn}

\begin{exa} Two extremely important example will be the frames $$\left(\A^1_{k\lser{t}},\P^1_{k\pow{t}},\widehat{\P}^1_{\cur{V}\pow{t}}\right)$$ and $$\left(\spec{k\lser{t}},\spec{k\pow{t}},\widehat{\A}^n_{\cur{V}\pow{t}}\right)$$ for $n\geq 0$. Here $\widehat{(\cdot)}$ denote the $\pi$-adic completion functor on schemes over $\cur{V}\pow{t}$. 
\end{exa}

Since we will now be exclusively we working with Huber's adic spaces, or equivalently Fujiwara/Kato's Zariski-Riemann spaces, we will henceforth (unless otherwise mentioned) use the word rigid space to mean a rigid space locally of finite type over $\spf{\cur{V}\pow{t}}^\rig$ in the sense of Definition II.2.2.18 and II.2.3.1 of \cite{rigidspaces}, or equivalently an adic space locally of finite type over $\mathrm{Spa}(S_K,\cur{V}\pow{t})$ in the sense of (1.1.2) and Definition 1.2.1 of \cite{huber2}, where $S_K=\cur{V}\pow{t}\otimes_\cur{V} K$. The equivalence of these two definitions is Theorem II.A.5.2. of \cite{rigidspaces}, and we will freely pass between the two interpretations, also note that this includes the notion of `classical' rigid spaces locally of finite type over $K$ in the sense of Tate. All rigid spaces we will consider shall be locally of finite type over $S_K$, and for any $S_K$-algebra $B$, topologically of finite type over $S_K$, denote by $B^+$ the integral closure of the image of $\cur{V}\pow{t}\tate{x_1,\ldots,x_n}$ inside $B$ for some presentation $S_K\tate{x_1,\ldots,x_n}\rightarrow B$. We will also denote $\mathrm{Spa}(B,B^+)\cong \mathrm{Spf}(B^+)^\rig$ by $\mathrm{Spa}(B)$, none of this depends on the choice of presentation. If $B=A_K:=A\otimes_{\cur{V}}K$ for some topologically finite type $\cur{V}\pow{t}$-algebra $A$, we will also write $A^+$ for the integral closure of $A$ inside $B$, thus $A^+=B^+$.

\begin{rem} It is worth noting that since $\cur{V}\pow{t}$ is Noetherian, thus we satisfy the hypothesis (1.1.1) of \cite{huber2} as well as being in the `t.u. rigid Noetherian' case of \cite{rigid spaces}.
\end{rem}

If $(X,Y,\mathfrak{P})$ is a frame then we will let $\mathfrak{P}_K=\mathfrak{P}^\rig$ denote the generic fibre of $\mathfrak{P}$, this is an rigid space of finite type over $\mathbb{D}^b_K:=\mathrm{Spa}(S_K,\cur{V}\pow{t})$. We will also let $P$ denote the mod-$\pi$ reduction of $\mathfrak{P}$, so that there is a homeomorphism $P\simeq\mathfrak{P}$. Then there is a specialisation map
$$
\mathrm{sp}:\mathfrak{P}_K\rightarrow \mathfrak{P}\simeq P
$$
as in \S II.3.1 of \cite{rigidspaces}, which locally on $\mathfrak{P}$ can be described as follows. If $\mathfrak{P}=\spf{A}$, then points of $\mathfrak{P}_K$ can be identified with certain valuations on $A_K$, and points of $\spf{A}$ with open prime ideals of $A$. Then the specialisation map sends $v:A_K\rightarrow \left\{0\right\}\cup\Gamma$ to the prime ideal consisting of elements $a\in A$ such that $v(a)<1$.

Let $[\mathfrak{P}_K]\subset \mathfrak{P}_K$ denote the subset of points whose corresponding valuation is of rank 1, by II.2.3.(c) and Proposition II.4.1.7 of \cite{rigidspaces} there is a map
$$
[\cdot]: \mathfrak{P}_K\rightarrow [\mathfrak{P}_K]
$$
which takes a point to its `maximal generisation'. The set $[\mathfrak{P}_K]$ is topologised via this quotient map, with respect to this topology it is Hausdorff, and $[\cdot]$ identifies $[\mathfrak{P}_K]$ with the maximal Hausdorff quotient of $\mathfrak{P}_K$ (Proposition II.2.3.9 of \cite{rigidspaces}). With respect to this topology, the inclusion $[\mathfrak{P}_K]\rightarrow \mathfrak{P}_K$ is \emph{not} continuous in general. 

This has the following local description. Suppose that $\mathfrak{P}=\spf{A}$ is affine, so that $\mathfrak{P}_K=\spa{A_K}$, the set of (equivalence classes of) continuous valuations $v$ on $A_K$ such that $v(a)\leq 1$ for every $a\in A^+$. Then $[\mathfrak{P}_K]=\sh{M}(A_K)$ is identified with the Berkovich spectrum of $A_K$, that is the set of (equivalence classes of) continuous rank 1 valuations $A_K\rightarrow \{0\}\cup \R^{>0}$ (although in \cite{rigidspaces}, this identification is only made for affinoid algebras over $K$, the same argument as given in Propositoon II.C.1.8 of \emph{loc. cit.} will work more generally). The map $\spa{A_K}\rightarrow \sh{M}(A_K)$ can then be described as follows. Let $v:A_K\rightarrow \left\{0 \right\}\cup\Gamma $ be a valuation, and let $I\subset A_K$ denote its support. Let $v:\mathrm{Frac}(A_K/I)\rightarrow \left\{ 0\right\}\cup\Gamma $ denote the induced valuation, and $V$ its valuation ring, with valuation ideal $P_v\subset V$. Then $\mathfrak{p} = \sqrt{(p)}\subset P_v$ is a height one prime ideal of $V$, and hence corresponds to a rank one valuation $v_\mathfrak{p}:A_K \rightarrow \left\{0 \right\}\cup\R^{>0}$. Then $[v]=v_\mathfrak{p}$.

If $Z\subset P$ a closed subset, we define
$$ ]Z[_\mathfrak{P}= \mathrm{sp}^{-1}(Z)^\circ.
$$
to be interior of the inverse image of $Z$ by the specialisation map. Exactly as in the previous section, if $\mathrm{sp}_{[\cdot]}:[\mathfrak{P}_K]\rightarrow P$ denotes the induced specialisation map on the subset of rank 1 points, then we have $]Z[_\mathfrak{P}=[\cdot]^{-1}(\mathrm{sp}_{[\cdot]}^{-1}(Z))$, and if locally we have $f_i\in \cur{O}_{\mathfrak{P}}$ whose reductions $\bar{f}_i$ define $Z$ inside $P$, then 
$$
]Z[_\mathfrak{P} = \left\{\left.x\in\mathfrak{P}_K \;\right|\; v_{[x]}(f_i)< 1 \;\forall i  \right\}
$$
(Proposition II.4.2.11 of \cite{rigidspaces}). Specialisation induces a continuous map $\mathrm{sp}_Z: ]Z[_\mathfrak{P}\rightarrow Z$. If $U\subset Z$ is open, then we set
$$
]U[_\mathfrak{P} = \overline{\mathrm{sp}_Y^{-1}(U)}.
$$
Again, we have $]U[_\mathfrak{P} = [\cdot]^{-1}(\mathrm{sp}_{[\cdot]}^{-1}(U))$ which shows that $]U[_\mathfrak{P}$ only depends on $U$ and $\mathfrak{P}$ (and not on $Z$) and if locally we have $g_j\in \cur{O}_\mathfrak{P}$ such that $U= Z\cap (\cup_j D(g_j) )$, then 
$$
]U[_\mathfrak{P} = \left\{\left. x\in]Z[_\mathfrak{P}  \;\right|\; \exists j \text{ s.t. }v_{[x]}(g_i) \geq 1  \right\}.
$$

\begin{rem} \label{interior} \begin{enumerate} 
\item We will often refer to $\mathrm{sp}_Y^{-1}(U)$ as the interior tube of $U$, and denote it by $]U[^\circ_\mathfrak{P}$. We do not know if it is literally the interior of $]U[_\mathfrak{P}$, however, this should not cause too much confusion.  
\item Since $\cur{O}_{\cur{E}_K} = \cur{V}\pow{t}\tate{t^{-1}}$, if $U$ is actually a scheme over $k\lser{t}$ then this interior tube is a rigid space locally of finite type over $\ek$. Moreover if $g_j$ are as above, then it can be described as
$$ \left\{\left.x\in ]Z[_\mathfrak{P} \;\right|\; \exists j\text{ s.t. }v_x(g_j)\geq 1  \right\}.
$$
In particular, if $\mathfrak{P}_{\cur{O}_{\ek}}$ denotes the base change of $\mathfrak{P}$ to $\cur{O}_{\ek}$, then for $U/k\lser{t}$, the interior tube of $U$ in $\mathfrak{P}$ and the interior tube of $U$ in $\mathfrak{P}_{\cur{O}_{\ek}}$ (defined in the obvious manner) are equal as rigid spaces over $\ek$.
\item Since $]U[_\mathfrak{P}=[\cdot]^{-1}(\mathrm{sp}_{[\cdot]}^{-1}(U))$ for any locally closed subscheme $U\subset P$, we can see that the formation of tubes behaves well with regard to unions and intersections of subschemes of $P$. For example, if $U=U_1\cup U_2$ is a union of closed subschemes $U_i$, then  $]U[_\mathfrak{P}=]U_1[_\mathfrak{P}\cup ]U_2[_\mathfrak{P}$, and if $U,V$ are subschemes of $P$ such that $U\cap V=\emptyset$, then $]U[_\mathfrak{P}\cap ]V[_\mathfrak{P} =\emptyset$. Note that neither of these is immediately obvious from the definitions. Another fact that follows along the same lines that we will need later on is that if we have a Cartesian diagram
$$
\xymatrix{ U' \ar[r]\ar[d] & \mathfrak{P}'\ar[d]^u \\ U \ar[r] & \mathfrak{P}}
$$
with horizontal arrows immersions, then $u_K^{-1}(]U[_\mathfrak{P})=]U'[_{\mathfrak{P}'}$.
\end{enumerate}
\end{rem}

If $(X,Y,\mathfrak{P})$ is a frame, we let $j:]X[_\mathfrak{P}\rightarrow ]Y[_\mathfrak{P}$ denote the inclusion. As in the previous section, for a sheaf $\sh{F}$ on $]Y[_\mathfrak{P}$ we define $j_X^\dagger\sh{F}:=j_*j^{-1}\sh{F}$.

\begin{defn} We define the rigid cohomology of the frame $(X,Y,\mathfrak{P})$ to be
$$H^i_\rig((X,Y,\mathfrak{P})/\cur{E}_K^\dagger):=H^i(]Y[_\mathfrak{P},j^\dagger_X\Omega^*_{]Y[_\mathfrak{P}/S_K})=H^i(]X[_\mathfrak{P},j^{-1}\Omega^*_{]Y[_\mathfrak{P}/S_K}).$$
\end{defn}

We will see shortly that these are indeed vector spaces over $\ekd$, thus justifying the notation. Beforehand, however, we will first discuss how the cofinal systems of neighbourhoods we saw in the previous section can be constructed entirely similarly in the context of frames over $\cur{V}\pow{t}$. 

So suppose that we have a frame $(X,Y,\mathfrak{P})$, with $\mathfrak{P}$ affine, and let $f_i,g_j\in\cur{O}_\mathfrak{P}$ be functions such that, letting $\bar{f}_i,\bar{g}_j\in\cur{O}_P$ denote their mod-$\pi$ reductions, we have
\begin{align*}
Y&= \cap_i V(\bar{f}_i) \subset P \\
X&= Y \cap \left(\cup_j D(\bar{g}_j)\right).
\end{align*}
Define
\begin{align*}
[Y]_n &= \left\{\left. x\in\mathfrak{P}_K \;\right|\; v_x(\pi^{-1}f_i^n)\leq 1\;\forall i \right\} \\
U_{m,j} &= \left\{\left. x\in ]Y[_\mathfrak{P} \;\right|\; v_x(\pi^{-1}g_j^m)\geq 1 \right\} \\
U_m &= \cup_j U_{m,j} \\
V_{n,{m},j} &= [Y]_n \cap U_{m,j} \\
V_{n,m} &= [Y]_n \cap U_m
\end{align*}
as in the previous section. Exactly as before, for $n,m\gg0$, these do not depend on the choice of $f_i,g_j$, and hence glue over an open affine cover of $\mathfrak{P}$. Moreover, we have $]Y[_\mathfrak{P}=\cup_n [Y_n]$. 
We will also need a slightly different version of the $U_m$ which better reflects the fact that we always have a non-trivial open immersion $X\rightarrow Y$. With this in mind, we choose $g'_j$ such that $X=Y\cap D(t)\cap \left(\cup_j D(g'_j)\right)$, and define
\begin{align*}
U'_{m,j} &= \left\{ \left. x\in ]Y[_\mathfrak{P} \;\right|\; v_x(\pi^{-1}g'^m_j)\geq 1,v_x(\pi^{-1}t^m)\geq 1 \right\} \\
U'_m &= \cup_j U_{m,j} \\
V'_{n,{m},j} &= [Y]_n \cap U_{m,j} \\
V'_{n,m} &= [Y]_n \cap U_m,
\end{align*}
again these do not depend on the choice of the $g'_j$ and hence glue over an open affine covering of $\mathfrak{P}$. Finally, for any increasing sequence of integers $\underline{m}(n)\rightarrow\infty$, we set
\begin{align*}
V_{\underline{m}}&=\bigcup_n V_{n,\underline{m}(n)} \\
V'_{\underline{m}}&=\bigcup_n V'_{n,\underline{m}(n)}.
\end{align*}

\begin{prop} \label{cofinal2} \begin{enumerate}
\item For all $n\geq0$, both $V_{n,m}$ and $V_{n,m}'$ form a cofinal system of neighbourhoods of $[Y]_n\;\cap\;]X[_\mathfrak{P}$ in $[Y]_n$.
\item As $\underline{m}$ varies, both $V_{\underline{m}}$ and $V'_{\underline{m}}$ form a cofinal system of neighbourhoods of $]X[_\mathfrak{P}$ in $]Y[_\mathfrak{P}$.
\end{enumerate}
\end{prop}

\begin{proof} Exactly the same argument that proves Lemma \ref{cofinalkey} and Proposition \ref{cofinal} works here.
\end{proof}

We can now prove that our notation for the rigid cohomology of a frame over $(X,Y,\mathfrak{P})$ is justified.

\begin{lem} The cohomology groups $H^i_\rig((X,Y,\mathfrak{P})/\cur{E}_K^\dagger)$ are vector spaces over $\cur{E}_K^\dagger$.
\end{lem}

\begin{proof} There is a morphism of frames
$$(X,Y,\mathfrak{P})\rightarrow \left(\spec{k\lser{t}},\spec{k\pow{t}},\mathrm{Spf}\left(\cur{V}\pow{t}\right)\right)$$
which induces a morphism of ringed spaces 
$$
(]Y[_\mathfrak{P},j_X^\dagger\cur{O}_{]Y[_\mathfrak{P}}) \rightarrow (\mathbb{D}^b_K, j_{\spec{k\lser{t}}}^\dagger\cur{O}_{\mathbb{D}^b_K})
$$
where $\mathbb{D}^b_K=\mathrm{Spa}(S_K,\cur{V}\pow{t})$. Now just observe that by Proposition \ref{cofinal2} and Lemma \ref{closed} that 
$$ \Gamma(\mathbb{D}^b_K, j_{\spec{k\lser{t}}}^\dagger\cur{O}_{\mathbb{D}^b_K}) = \mathrm{colim}_m \frac{S_K\tate{T}}{(\pi-t^mT)} = \ekd $$
and the lemma follows.
\end{proof}

\section{Sundry properties of rigid spaces and morphisms between them}\label{sundry}

In this section we collect together a few technical results we will need about rigid spaces and morphisms between them, and as such it can be safely skimmed and the results referred back to as necessary. There are certain properties of morphisms of rigid spaces, that we will need to use, such as finite, proper, etc., which are defined both by Huber in \cite{huber2} and by Fujiwara and Kato in \cite{rigidspaces}. Since results proved both in \cite{huber2} and in \cite{rigidspaces} will be useful for us, it will be necessary to know that the two definitions coincide. Thus part of this section is devoted to proving these equivalences. We will also need a results concerning the \'{e}tale locus of a morphism of rigid spaces. First, however, we will prove a few results about the support of coherent sheaves on rigid spaces, and about the interaction of closed analytic subspaces with the kinds of open subspaces considered in the previous section. Unless otherwise mentioned, all rigid spaces will be assumed to be locally of finite type over $\mathbb{D}^b_K$.

\begin{defn} Let $\sh{X}$ be a rigid space. Then a closed analytic subspace of $\sh{X}$ is a subspace defined by a coherent sheaf of ideals $\sh{I}\subset \cur{O}_{\sh{X}}$. This is again a rigid space, with structure sheaf given by $\cur{O}_{\sh{X}}/\sh{I}$.
\end{defn}

\begin{rem} Note that by Proposition II.7.3.5 of \cite{rigidspaces} a closed analytic subspace of $\sh{X}$ is exactly the image of a closed immersion of rigid spaces in the sense of Definition II.7.3.7 of \emph{loc. cit}. 
\end{rem}

\begin{prop} \label{support} Let $\sh{F}$ be a coherent sheaf on a rigid space $\sh{X}$. Then the support $\mathrm{supp}(\sh{F})$ of $\sh{F}$ is contained in a closed analytic subspace of $\sh{X}$.
\end{prop}

\begin{proof} The question is local on $\sh{X}$, which we may thus assume to be affine $\sh{X}\cong \mathrm{Spa}(B)$ for some topologically finite type $S_K$-algebra $B$. Then $\sh{F}$ is the $\cur{O}_\sh{X}$-module associated to some finite $B$-module $M$. If $\sh{F}^\mathrm{alg}$ is the coherent sheaf on $X=\spec{A}$ associated to $M$, and $\varphi: \sh{X}\rightarrow X$ the canonical morphism of ringed spaces, then $\sh{F}\cong \varphi^*\sh{F}^\mathrm{alg}$ (see for example \S II.6.6 of \cite{rigidspaces}). Hence 
$$
\mathrm{supp}(\sh{F})\subset \varphi^{-1}( \mathrm{supp}(\sh{F}^\mathrm{alg}))
$$ 
and $\mathrm{supp}(\sh{F}^\mathrm{alg})$ is the closed subset $V(I)$ of $\spec{A}$ defined by the ideal $I=\mathrm{Ann}(M)\subset A$. By Proposition II.7.3.16 of \cite{rigidspaces}, this inverse image coincides with the closed analytic subspace of $\mathrm{Spa}(A)$ corresponding to $I$.
\end{proof}

\begin{prop} \label{overconv} Let $\sh{X}$ be a quasi-compact rigid space, and $f\in\Gamma(\sh{X},\cur{O}_\sh{X})$. Let $T\subset \sh{X}$ be a closed analytic subspace, and suppose that 
$$
T\cap \left\{\left.x\in\sh{X}\;\right|\; v_x(f)\geq 1 \right\} =\emptyset.
$$
Then there exists some $m$ such that 
$$
T\cap \left\{\left.x\in\sh{X}\;\right|\; v_x(\pi^{-1}f^m)\geq 1 \right\} =\emptyset.
$$
\end{prop}

\begin{proof} The question is local, so we may assume that $\sh{X}$, and hence $T$, is affinoid, say $T\cong \mathrm{Spa}(B)$. Let $g\in B$ be the pullback of $f$, we are required to show that
$$ v_x(g) < 1\; \forall x\in T \Rightarrow \exists m\text{ s.t. }v_x(\pi^{-1}g^m) < 1\;\forall x\in T.
$$
By Lemma \ref{vals} (or rather, its analogue for rigid varieties over $S_K$, the proof goes over verbatim) we may restrict to height one points $x\in [T]$. But now this can be rephrased as
$$
v_x(g) < 1\; \forall x\in [T] \Rightarrow \exists m\text{ s.t. }v_x(g) < r^{-1/m} \;\forall x\in [T],
$$
so if we let $\Norm{\cdot}_\mathrm{sup}=\mathrm{sup}_{v\in\sh{M}(B)} v(\cdot)$ denote the spectral semi-norm on $B$, then it suffices to show that 
$$
v(g)<1\; \forall v\in\sh{M}(B)\Rightarrow \Norm{g}_\mathrm{sup} <1.
$$
This then follows from compactness of $\sh{M}(B)$.
\end{proof}

\begin{rem} This proposition is closely related to the fact that the underlying set of a closed analytic subspace of $\sh{X}$ is an overconvergent closed subset of $\sh{X}$, i.e. the inverse image of a closed subset of $[\sh{X}]$.
\end{rem}

Now we turn to proving the equivalences we require between the definitions of separatedness, properness and finiteness given by Huber and Fujiwara/Kato.

\begin{defn} A morphism $\sh{X}\rightarrow \sh{Y}$ of rigid spaces is separated if the diagonal morphism
$$
\Delta:\sh{X}\rightarrow \sh{X}\times_{\sh{Y}}\sh{X}
$$
is a closed immersion. Note that this is the definition given in both \cite{rigidspaces} and \cite{huber2}. 
\end{defn}

\begin{defn} From now on `rigid variety over $S_K$' will mean `rigid space separated and locally of finite type over $\mathbb{D}^b_K$'.
\end{defn}

Note that by Corollary II.7.5.12 (3) of \cite{rigidspaces}, any morphism between rigid varieties is separated. Any morphism is also locally of finite type in the sense of Definition II.2.3.1 of \cite{rigidspaces} by Proposition II.2.3.2 of \cite{rigidspaces}, and in the sense of Definition 1.2.1 of \cite{huber2} by Lemma 3.5 (iv) of \cite{huber}. A morphism of rigid varieties over $S_K$ is of finite type (in the sense of either \cite{rigidspaces} or \cite{huber2}) if and only if it is quasi-compact. Indeed, this follows by Proposition II.7.1.5 (1) of \cite{rigidspaces} and is the definition of finite type in \cite{huber2}.

\begin{defn} A morphism $f:\sh{X}\rightarrow\sh{Y}$ of finite type between rigid varieties over $S_K$ is said to be:
\begin{enumerate} \item proper in the sense of Huber if for any morphism $\sh{Z}\rightarrow \sh{Y}$ of adic spaces (not necessarily locally of finite type over $\mathbb{D}^b_K$) the map 
$$
\sh{X}\times_{\sh{Y}}\sh{Z}\rightarrow \sh{Z}
$$
is closed.
\item proper in the sense of Fujiwara/Kato if for any morphism $\sh{Z}\rightarrow \sh{Y}$ of rigid spaces (not necessarily locally of finite type of $\mathbb{D}^b_K$) the map 
$$
\sh{X}\times_{\sh{Y}}\sh{Z}\rightarrow \sh{Z}
$$
is closed.
\end{enumerate}
\end{defn}

Note that \emph{a priori} the two definitions are not the same, since the category of objects we are base changing by could be different. However, by Corollary II.7.5.16 of \cite{rigidspaces} it suffices to check the universal closedness for Fujiwara/Kato properness for $\sh{Z}= \mathbb{D}_\sh{Y}^n$ for $n\geq1$, and hence Huber properness implies Fujiwara/Kato properness.

To show the converse, we first recall some notation. For an adic formal scheme $\mathfrak{X}$ of finite ideal type, not necessarily of finite type over $\cur{V}\pow{t}$, we let $\mathfrak{X}^\rig$ denote the associated coherent rigid space in the sense of \S II.2.1 of \cite{rigidspaces}. For a rigid variety $\sh{X}$ over $S_K$ and $x\in\sh{X}$ we let $A_x=\cur{O}^+_{\sh{X},x}$ denote the stalk of the integral structure sheaf of $\sh{X}$ at $x$, $B_x=\cur{O}_{\sh{X},x}=A_x[\pi^{-1}]$ the stalk of the structure sheaf, $\mathfrak{m}_x$ the maximal idea of $B_x$, and $K_x$ the residue field $B_x/\mathfrak{m}_x$. Corollary II.3.2.8 of \cite{rigidspaces} tells us that $\mathfrak{m}_x\subset A_x$, and that $V_x:=A_x/\mathfrak{m}_x$ is a valuation ring inside $K_x$. Let $k_x$ denote the residue field of $V_x$. 

Following \S1.1 of \cite{huber2}, we define an affinoid field to be a pair $(A^\triangleright,A^+)$ where $A^+$ is a valuation ring with quotient field $A^\triangleright$, and the valuation topology on $A^\triangleright$ is induced by a valuation of rank 1. Following \S II.3.3 of \cite{rigidspaces}, we define a rigid point of a rigid variety $\sh{X}$ over $S_K$ to be a morphism of rigid spaces $\spf{V}^\rig\rightarrow \sh{X}$ (not necessarily of finite type over $\mathbb{D}^b_K$) where $V$ is an $a$-adically complete valuation ring for some $a\in\mathfrak{m}_V\setminus \left\{ 0\right\}$. 

\begin{lem} If $f:\sh{X}\rightarrow \sh{Y}$ is Fujiwara/Kato proper, then it is Huber proper.
\end{lem}

\begin{proof} This more or less follows from the respective valuative criteria, i.e. Theorem II.7.5.17 of \cite{rigidspaces} and Lemma 1.3.10 of \cite{huber2}, with a little care taken to ensure that rigid points of $\sh{X}$ correspond to morphisms $\mathrm{Spa}(A^\triangleright,A^+)\rightarrow \sh{X}$ where $(A^\triangleright,A^+)$ is an affinoid field.

So suppose that we have a diagram of adic spaces
$$ \xymatrix{ \mathrm{Spa}(A^\triangleright,A^+) \ar[r]\ar[d] & \sh{X}  \ar[d] \\
\mathrm{Spa}(B^\triangleright, B^+) \ar[r] & \sh{Y}
}
$$
with $\sh{X}\rightarrow \sh{Y}$ Fujiwara/Kato proper, such that $\widehat{A}^\triangleright=\widehat{B}^\triangleright$, we must show that there is a unique morphism $\mathrm{Spa}(B^\triangleright,B^+)\rightarrow \sh{X}$ making the diagram commute. Actually, since $\sh{X}\rightarrow \sh{Y}$ is separated, by the valuative criterion for separatedness (i.e. Lemma 1.3.10 of \cite{huber2}) it suffices to show that there exists such a morphism. Note that since $ \mathrm{Spa}(A^\triangleright,A^+)$ only depends on the completion $(\widehat{A}^\triangleright,\widehat{A}^+)$ we may assume that $A^\triangleright=B^\triangleright$ is complete, and we have valuation rings $B^+\subset A^+ \subset A^\triangleright$. 

Then by 1.1.8 of \cite{huber2}, the morphism $ \mathrm{Spa}(A^\triangleright,A^+) \rightarrow \sh{X}$ corresponds to a pair $(x,\varphi)$ where $x\in \sh{X}$ and $\varphi:K_x\rightarrow A^\triangleright$ is a continuous homomorphism such that $V_x=\varphi^{-1}(A^+)$. Hence this extends uniquely to a morphism $\widehat{V}_x\rightarrow A^+$ of complete valuation rings (where $\widehat{V}_x$ is the $\pi$-adic completion of $V_x$) and hence a rigid point $ \spf{A^+}^\rig\rightarrow \sh{X}$, similarly the morphism $\mathrm{Spa}(A^\triangleright, B^+) \rightarrow \sh{Y}$ corresponds to a rigid point of $\sh{Y}$, and there is a commutative diagram
$$ \xymatrix{ \spf{A^+}^\rig \ar[r]\ar[d] & \sh{X} \ar[d] \\
\spf{B^+}^\rig \ar[r] & \sh{Y}.
}
$$
Since the morphism $B^+\rightarrow A^+$ is a localisation at a prime ideal of $B^+$, the morphism $$\spf{A^+}^\rig\rightarrow \spf{B^+}^\rig$$ is a generisation in the sense of II.7.5(c) of \cite{rigidspaces}, and hence there exists a unique morphism $\spf{B^+}^\rig\rightarrow \sh{X}$ making the diagram commute. Again, the rigid point $\spf{B^+}^\rig\rightarrow \sh{X}$ corresponds to a point $x\in\sh{X}$ and a continuous homomorphism $\widehat{V}_x\rightarrow B^+$ of complete valuation rings, and hence a continuous homomorphism $V_x\rightarrow B^+$. This extends uniquely to a continuous homomorphism $K_x\rightarrow A^\triangleright$ and hence a morphism $\mathrm{Spa}(A^\triangleright,B^+) \rightarrow \sh{X}$ as required.
\end{proof}

Henceforth we will simply refer to a morphism $f:\sh{X}\rightarrow \sh{Y}$ of rigid varieties over $S_K$ being proper. 

\begin{defn} A morphism $f:\sh{X}\rightarrow \sh{Y}$ of rigid analytic varieties over $S_K$ is said to be:
\begin{enumerate} \item finite in the sense of Huber if locally on $\sh{Y}$ it is of the form
$$\mathrm{Spa}(B)\rightarrow \mathrm{Spa}(A)$$
for some finite morphism $A\rightarrow B$ of topologically finite type $S_K$-algebras $A,B$;
\item finite in the sense of Fujiwara/Kato if locally on $\sh{Y}$ it arises as the generic fibre of a finite morphism $\mathfrak{X}\rightarrow \mathfrak{Y}$ between formal schemes of finite type over $\cur{V}\pow{t}$.
\end{enumerate}
\end{defn}

It is clear that Fujiwara/Kato finiteness implies Huber finiteness.

\begin{lem} If $f:\sh{X}\rightarrow \sh{Y}$ is Huber finite, then it is Fujiwara/Kato finite.
\end{lem}

\begin{proof} We may suppose that $f$ is associated to a finite morphism $A\rightarrow B$ of topologically finite type $S_K$-algebras. Let $A^+\rightarrow B^+$ be the associated morphism of $+$-parts, note that $B^+$ is the integral closure of $A^+$ in $B$. Choose $b_1,\ldots,b_n\in B^+$ which generate $B$ as an $A$-module, and hence topologically as an $A$-algebra. Since each $b_i$ is integral over $A^+$,
$$A^+\rightarrow A^+[b_1,\ldots,b_n]$$
is a finite formal model for $A\rightarrow B$.
\end{proof}

Henceforth we will simply refer to a morphism $f:\sh{X}\rightarrow \sh{Y}$ of rigid varieties over $S_K$ being finite. Having proved the required equivalence between the different definitions of properness and finiteness, we move on to the second main result of this section, concerning the openness of the \'{e}tale locus of a morphism of rigid varieties over $S_K$. 

\begin{defn} Let $f:\sh{X}\rightarrow \sh{Y}$ be a morphism between rigid spaces over $S_K$, $\Delta:\sh{X}\rightarrow \sh{X}\times_{\sh{Y}}\sh{X}$ the closed immersion defined by the diagonal, and $\sh{I}= \ker(\cur{O}_{\sh{X}\times_{\sh{Y}}\sh{X}}\rightarrow \Delta_*\cur{O}_{\sh{X}})$ the kernel of the multiplication map. Then we define the module of differentials 
$$ \Omega^1_{\sh{X}/\sh{Y}} = \Delta^*(\sh{I})=\sh{I}/\sh{I}^2.$$
by 1.6 of \cite{huber2} this is a coherent $\cur{O}_{\sh{X}}$-module,
\end{defn}

While the following definitions are not those given in \S1 of \cite{huber2}, the results there show that they are equivalent.

\begin{defn} A morphism $f:\sh{X}\rightarrow \sh{Y}$ of rigid varieties over $S_K$ is said to be:
\begin{enumerate} 
\item unramified if $\Omega^1_{\sh{X}/\sh{Y}}=0$;
\item flat if for each $x\in \sh{X}$, $\cur{O}_{X,x}$ is flat over $\cur{O}_{Y,f(x)}$;
\item \'{e}tale if it is flat and unramified.
\end{enumerate}
\end{defn}

It follows immediately from Proposition \ref{support} that the locus where a morphism $f:\sh{X}\rightarrow \sh{Y}$ is not unramified is a closed analytic subspace of $\sh{X}$. The following result says that the same is true for \'{e}taleness.

\begin{prop} \label{openetale} Let $f:\sh{X}\rightarrow \sh{Y}$ be a morphism of rigid varieties over $S_K$. Then $f$ is \'{e}tale away from a closed analytic subspace of $\sh{X}$.
\end{prop}

\begin{proof} We may assume that $f$ is unramified, and the question is local on both $\sh{X}$ and $\sh{Y}$, which we may thus assume to be affinoid. Hence by Proposition 1.6.8 of \cite{huber2} we may factor $\sh{X}\rightarrow \sh{Y}$ as a closed immersion $g:\sh{X}\rightarrow \sh{Z}$ follows by an \'{e}tale map $h:\sh{Z}\rightarrow \sh{Y}$. Thus it suffices to show that the closed immersion $g:\sh{X}\rightarrow \sh{Z}$ is an isomorphism away from a closed analytic subspace of $\sh{Z}$. But now this just follows from the fact that $g$ is an isomorphism away from the support of the coherent sheaf
$$
\ker(\cur{O}_{\sh{Z}}\rightarrow g_*\cur{O}_{\sh{X}})
$$ together with Proposition \ref{support}.
\end{proof}

\begin{rem} Note that this closed subspace will in general be the whole of $\sh{X}$, and this result will only be useful when we already know that $f$ is generically \'{e}tale, i.e. \'{e}tale on some open subset of $\sh{X}$.
\end{rem}

\section{Independence of the frame}

In this section we prove that for a smooth and proper frame $(X,Y,\mathfrak{P})$, the rigid cohomology
$$
H^i_\rig((X,Y,\mathfrak{P})/\ekd)
$$
only depends on $X$, and gives rise to a functor
$$
H^i_\rig(X/\ekd)
$$
on the category of embeddable varieties. (It is relatively easy to then extend this to non-embeddable varieties, we will do this in the sequel \cite{rclsf3}). We will follow closely Berthlot's original proof of independence for rigid cohomology in \cite{rigcohber} and \cite{iso}, the key results being the Strong Fibration Theorem and the overconvergent Poincar\'{e} Lemma below. 

\begin{prop}[Strong Fibration Theorem] \label{strongprop} Suppose that 
$$
\xymatrix{
& Y'  \ar[r]^{i'}\ar[d]^v & \mathfrak{P}' \ar[d]^u \\
X \ar[ur]^{j'} \ar[r]^j & Y \ar[r]^i & \mathfrak{P}
}
$$
is a diagram of frames over $\cur{V}\pow{t}$, such that $v$ is proper, and $u$ is \'{e}tale in a neighbourhood of $X$. Then $u_K$ induces an isomorphism between a cofinal system of neighbourhoods of $]X[_{\mathfrak{P}}$ in $]Y[_{\mathfrak{P}}$ and a cofinal system of neighbourhoods of $]X[_{\mathfrak{P}'}$ in $]Y'[_{\mathfrak{P}'}$. 
\end{prop}

\begin{proof} We follow closely the proof of Th\'{e}or\`{e}me 1.3.5 of \cite{iso}. We may replace $Y'$ by the closure of $X$ in $Y'$, and hence assume that $u^{-1}(X)\cap Y'=X$. Using the standard neighbourhoods $V_{\underline{m}}$ constructed in the previous section, it therefore suffices to prove that for all $n\gg0$ there exists some $d\geq n$, such that
$$ [Y']_{d} \cap u_K^{-1}([Y]_n \cap U_m )\rightarrow [Y]_n \cap U_m
$$
is an isomorphism for $m\gg0$. The question is local on $\mathfrak{P}$, which we may thus assume to be affine, isomorphic to $\spf{A}$. After base changing to $[Y]_n$ we may assume that $[Y]_n=\mathfrak{P}_K$ and that the closed immersion $Y\rightarrow P$ of $Y$ into the reduction of $\mathfrak{P}$ is nilpotent. The question is also local on $X$, which we may thus assume to be of the form $D(tg)\cap Y$ for some $g\in A$. Thus we have
$$
U_m=\left\{\left.x\in \mathfrak{P}_K \;\right|\; v_x(p^{-1}(tg)^m)\geq 1 \right\}
$$
and we wish to show that there exists some $d\geq n$ such that
$$
[Y']_d \cap u_K^{-1}(U_m) \isomto U_m
$$
for $m\gg0$. Let $P$ and $P'$ denote the reductions of $\mathfrak{P}$ and $\mathfrak{P}'$ respectively, and let $U'_m=u_K^{-1}(U_m)$. Write $P'_Y=Y\times_P P'$ and $P'_X=X\times_{P}P'$, note that $P'_Y$ has the same underlying space as $P'$. Since $P'_X\rightarrow X$ admits a section around which it is \'{e}tale, it follows that $X$ is open and closed in $P'_X$, and since $P'_X$ is open in $P'_Y$, $X$ must be open in $P'_Y$.

Let $D=P'_Y\setminus X$ be the closed complement, since $X\subset Y'$ it follows that $P'_Y$ is the union of its two closed subschemes $Y'$ and $D$, and $P'_X$ is the union of its two components $X$ and $P'_X\setminus X = D\cap u^{-1}(X)$. Thus we have
\begin{align*}
\mathfrak{P}'_K &= ]Y'[_{\mathfrak{P}'}\; \cup \;]D[_{\mathfrak{P}'} \\
]P'_X[_{\frak{P}'}=u_K^{-1}(]X[_\mathfrak{P}) &= ]X[_{\mathfrak{P}'} \;\cup\; (]D[_{\mathfrak{P}'}\; \cap\; u_K^{-1}(]X[_\mathfrak{P}))
\end{align*}
the first being an open covering and the second being a decomposition into components. Arguing by quasi-compactness, we can see that there must be some $d,l$ such that 
$$
\mathfrak{P}'_K =[Y']_{d}\cup [D]_{l}.
$$
Since $u^{-1}(X)\cap Y'=X$, it follows that 
$$]Y'[_{\mathfrak{P}'} \;\cap\; ]D[_{\mathfrak{P}'}\;\cap \;u_K^{-1}(]X[_{\mathfrak{P}}) =]X[_{\mathfrak{P}'}\cap ]D[_{\mathfrak{P}}'=]X\cap D[_{\frak{P}'}=\emptyset,$$
so \emph{a fortiori} $[Y']_{d} \cap [D]_l \cap u_K^{-1}(]X[_\mathfrak{P})=\emptyset$. By the maximum principle applied on the separated quotient of $[Y']_{d} \cap [D]_l$ together with Lemma \ref{vals}, we can see that we must in fact have
$$ [Y']_{d}\cap [D]_l\cap U'_m=\emptyset
$$ 
for $m\gg0$. Thus $U'_m = ([Y']_d \cap U'_m)\cup([D]_l\cap U'_m)$ is a decomposition of $U'_m$ into components.

Define $T_K^m=[Y']_d\cap U'_m$, this is an open and closed subset of $U'_m$, and hence a quasi-compact rigid space over $S_K$. By Remark \ref{interior}(2) and the weak fibration theorem (Proposition 1.3.1 of \cite{iso}), $T_K^m\rightarrow U_m$ induces an isomorphism between the interior tubes $]X[^\circ_{\mathfrak{P}'}$ and $]X[_\mathfrak{P}^\circ$. By Proposition \ref{openetale}, the locus where $T_K^m\rightarrow U_m$ is not \'{e}tale is a closed analytic subset of $T_K^m$, and the fact that it is \'{e}tale on the interior tubes together with Proposition \ref{overconv} implies that $T_K^m\rightarrow U_m$ is \'{e}tale for all $m\gg0$. Since $u_K$ is proper, we must also have that $T_K^m\rightarrow U_m$ is proper, thus by Proposition 1.5.5 of \cite{huber2} it is finite.

Proposition II.7.2.4 of \cite{rigidspaces} now shows that (again, for $m\gg0$) $T_K^m\rightarrow U_m$ is the morphism associated to a coherent $\cur{O}_{U_m}$-algebra $\sh{A}$, say, and by the `classical' weak fibration theorem we know that this morphism is an isomorphism on the interior tube $]X[_\mathfrak{P}^\circ$. Hence by applying Proposition \ref{support} to the kernel and cokernel of $\cur{O}_{U_m}\rightarrow \cur{A}$, and then using Proposition \ref{overconv} we can see that it is an isomorphism on $U_m$ for $m\gg0$. This completes the proof.
\end{proof}

To be able to use this, we will need to know that, locally, a smooth morphism of frames $(X,Y,\mathfrak{P}')\rightarrow (X,Y,\mathfrak{P})$ factors into an \'{e}tale morphism of frames followed by a projection $(X,Y,\widehat{\A}^d_\mathfrak{P})\rightarrow (X,Y,\mathfrak{P})$, where $Y$ is embedded in $\widehat{\A}^d_\mathfrak{P}$ via the zero section. This is the content of the following lemma. 

\begin{lem} \label{frameetale} Let 
$$
\xymatrix{
& Y'  \ar[r]^{i'}\ar[d]^v & \mathfrak{P}' \ar[d]^u \\
X \ar[ur]^{j'} \ar[r]^j & Y \ar[r]^i & \mathfrak{P}
}
$$
be a diagram of smooth frames over $\cur{V}\pow{t}$, such that $u$ is smooth in a neighbourhood of $X$. Let $\sh{I}'\subset \cur{O}_{\mathfrak{P}'}$ denote the ideal of $Y'$ in $\mathfrak{P}'$, and $I'$ that of $Y'$ inside $P'_Y=\mathfrak{P}'\times_{\mathfrak{P}} Y$. Suppose that there are sections $t_1,\ldots,t_d\in\Gamma(\mathfrak{P}',\sh{I}')$ inducing a basis $\bar{t}_1,\ldots,\bar{t}_d$ of the conormal sheaf $I'/I'^2$ in a neighbourhood of $X$. Then the morphism $\varphi:\mathfrak{P}'\rightarrow \widehat{\A}^d_\mathfrak{P}$ defined by $t_1,\ldots,t_d$ maps $Y'$ into $Y$ and is \'{e}tale in a neighbourhood of $X$.
\end{lem}

\begin{proof} The proof is identical to that of Th\'{e}or\`{e}me 1.3.7 in \cite{iso}.
\end{proof}

The other fundamental result that we require is a suitable version of the Poincar\'{e} Lemma.

\begin{prop}[Poincar\'{e} Lemma] \label{poincare} Let $(X,Y,\mathfrak{P})$ be a smooth frame over $\cur{V}\pow{t}$, and let
$$
u:  (X,Y,\widehat{\A}^d_\mathfrak{P})\rightarrow(X,Y,\mathfrak{P})
$$
be the natural morphism of frames. Then the induced morphism
$$
j^\dagger _X\cur{O}_{]Y[_\mathfrak{P}} \rightarrow \mathbf{R}u_{K*} j_X^\dagger \Omega^*_{]Y[_{\widehat{\A}^d_\mathfrak{P}}/]Y[_\mathfrak{P}}
$$
is a quasi-isomorphism. Hence the induced morphism
$$ j^\dagger _X\Omega^*_{]Y[_\mathfrak{P}/S_K} \rightarrow \mathbf{R}u_{K*} j_X^\dagger \Omega^*_{]Y[_{\widehat{\A}^d_\mathfrak{P}}/S_K}
$$
is a quasi-isomorphism.
\end{prop}

\begin{proof} The question is local on $\mathfrak{P}$, which we may thus assume to be affine, $\mathfrak{P}\cong \spf{A}$. Let $f_i\in A$ be functions whose reductions $\bar{f}_i$ define the ideal of $Y$ inside $P$ and choose $g_j\in A$ such that $X= (D(t)\cap (\bigcup_j D(\bar{g}_j)))\cap Y$.

First suppose that $d=1$, and let $\mathbb{D}^1_{S_K}$ denote the unit disk $(\widehat{\A}^1_{\cur{V}\pow{t}})_K=\mathrm{Spa}(S_K\tate{X})$ over $S_K$, with co-ordinate $X$. If we let $[Y]_n=\left\{\left.x\in \mathfrak{P}_K  \;\right|\; v_x(\pi^{-1}f_i^n)\leq1 \;\forall i \right\}$, then $]Y[_\mathfrak{P}=\cup_n [Y]_n$, and we may base change to $[Y]_n$ and hence assume that $]Y[_\mathfrak{P}=\mathfrak{P}_K$. Now define 
\begin{align*}
[Y]'_n &= \left\{\left.x\in \mathfrak{P}_K \times_{S_K} \mathbb{D}^1_{S_K} \;\right|\; v_x(\pi^{-1}X^n)\leq 1 \right\} \\
U_m' &= \left\{\left. x\in ]Y[_{\widehat{\A}^1_{\mathfrak{P}}}  \;\right|\; v_x(\pi^{-1}t^m) \geq 1,\; \exists j \text{ s.t. }v_x(\pi^{-1}g_j^m) \geq 1  \right\}\\
U_m &=  \left\{\left. x\in \mathfrak{P}_K  \;\right|\; v_x(\pi^{-1}t^m) \geq 1,\; \exists j \text{ s.t. }v_x(\pi^{-1}g_j^m) \geq 1  \right\} \\
U_{m,j}' &= \left\{\left. x\in ]Y[_{\widehat{\A}^1_{\mathfrak{P}}}  \;\right|\; v_x(\pi^{-1}t^m) \geq 1,\; v_x(\pi^{-1}g_j^m) \geq 1  \right\}\\
U_{m,j} &=  \left\{\left. x\in \mathfrak{P}_K  \;\right|\; v_x(\pi^{-1}t^m) \geq 1,\; v_x(\pi^{-1}g_j^m) \geq 1  \right\} \\
\end{align*}
so that, by Proposition \ref{cofinal2} and the preceding discussion, $]Y[_{\widehat{\A}^1_\mathfrak{P}}=\cup_n[Y]'_n$, and $[Y]'_n\;\cap\; U'_m$ (resp. $U_m$) is a cofinal system of neighbourhoods of $]X[_{\widehat{\A}^1_\mathfrak{P}}\;\cap\; [Y]'_n$ in $[Y]'_n$ (resp. $]X[_\mathfrak{P}$ in $\mathfrak{P}_K$). To prove the claim, it suffices to prove it after base changing to each $U_{m_0,j}$ for some fixed $m_0$, hence we may assume that there is only one $g_j$, or in other words that there exists some $g$ such that 
\begin{align*}
U_m' &= \left\{\left. x\in ]Y[_{\widehat{\A}^1_{\mathfrak{P}}}  \;\right|\; v_x(\pi^{-1}t^m) \geq 1,\; v_x(\pi^{-1}g^m) \geq 1  \right\}\\
U_m &=  \left\{\left. x\in \mathfrak{P}_K  \;\right|\; v_x(\pi^{-1}t^m) \geq 1,\; v_x(\pi^{-1}g^m) \geq 1  \right\} .
\end{align*}
Hence each $[Y]'_n\cap U_m'$ (resp. $U_m$) is affinoid, and if we let $B'_{n,m}$ (resp. $B_m$) denote the affinoid algebra over $S_K$ corresponding to $[Y]'_n\cap U_m'$ (resp. $U_m$), then we have
$$
B'_{n,m}\cong B_m \tate{r^{1/n}X} := \frac{B_m\tate{X,T}}{(\pi T-X^n)}.
$$ 
We next claim that the result holds for global sections on $\mathfrak{P}_K$, that is 
$$
\mathrm{colim}_m B_m\rightarrow \mathbf{R}\Gamma(\mathfrak{P}_K,\mathbf{R}u_{K*}j_X^\dagger \Omega^*_{]Y[_{\widehat{\A}^1_\mathfrak{P}}/]Y[_\mathfrak{P}})
$$ is an isomorphism. Indeed, letting $j_{m,n}: [Y]_n' \cap U_m' \rightarrow [Y]_n'$ denote the natural inclusion we have a sequence of quasi-isomorphisms
\begin{align*}
 \mathbf{R}\Gamma(\mathfrak{P}_K,\mathbf{R}u_{K*}j_X^\dagger \Omega^*_{]Y[_{\widehat{\A}^d_\mathfrak{P}}/\mathfrak{P}_K})  &\cong \mathbf{R}\Gamma(]Y[_{\widehat{\A}^1_\mathfrak{P}},j_X^\dagger \Omega^*_{]Y[_{\widehat{\A}^1_\mathfrak{P}}/]Y[_\mathfrak{P}}) \\
 &\cong \mathbf{R}\mathrm{lim}_n \mathbf{R}\Gamma([Y]'_n , j_X^\dagger \Omega^*_{]Y[_{\widehat{\A}^1_\mathfrak{P}}/]Y[_\mathfrak{P}}|_{[Y]'_n}) \\
 &\cong \mathbf{R}\mathrm{lim}_n \mathbf{R}\Gamma([Y]'_n , \mathrm{colim}_m j_{m,n*}\Omega^*_{[Y]'_n\cap U_m'/U_m})  \\
 &\cong\mathbf{R}\mathrm{lim}_n \mathrm{colim}_m \mathbf{R}\Gamma([Y]'_n , \mathbf{R} j_{m,n*}\Omega^*_{[Y]'_n\cap U_m'/U_m})  \\ 
 &\cong \mathbf{R}\mathrm{lim}_n \mathrm{colim}_m \mathbf{R}\Gamma([Y]'_n\cap U_m' , \Omega^*_{[Y]'_n\cap U'_m/U_m} ) \\
 &\cong \mathbf{R}\mathrm{lim}_n \mathrm{colim}_m \Gamma([Y]'_n\cap U_m' , \Omega^*_{[Y]'_n\cap U'_m/U_m} ) \\
 &\cong  \mathbf{R}\mathrm{lim}_n \mathrm{colim}_m (B_m\tate{r^{1/n}X}\rightarrow B_m\tate{r^{1/n}X}dX).
\end{align*}
Thus what we want to prove is that 
$$
\mathrm{colim}_m B_m\rightarrow \mathbf{R}\mathrm{lim}_n \mathrm{colim}_m (B_m\tate{r^{1/n}X} \rightarrow B_m\tate{r^{1/n}X} dX)
$$
is a quasi-isomorphism. Write
\begin{align*}
H^i_n &= \mathrm{colim}_mH^i(B_m\tate{r^{1/n}X} \rightarrow B_m\tate{r^{1/n}X} dX) \\
 H^j &= H^j(\mathbf{R}\mathrm{lim}_n \mathrm{colim}_m (B_m\tate{r^{1/n}X} \rightarrow B_m\tate{r^{1/n}X} dX)) 
\end{align*}
so that we have
\begin{align*}
H^0  &\cong \mathrm{lim}_n H^0_n \\
H^2 &\cong \mathrm{lim}^1_n H^1_n
\end{align*}
and an exact sequence
$$
0\rightarrow \mathrm{lim}_n H^1  \rightarrow H^1 \rightarrow \mathrm{lim}^1_n H^0 \rightarrow 0
$$
and $H^j=0$ for $j\neq0,1,2$. Since $H^0_n = \mathrm{colim}_m B_m$ for all $n$, and the transition maps $H^1_n\rightarrow H^1_{n-1}$ are all zero, it follows that $H^0=\mathrm{colim}_m B_m$, $H^1=H^2=0$ and hence
$$
\mathrm{colim}_m B_m\rightarrow \mathbf{R}\Gamma(\mathfrak{P}_K,\mathbf{R}u_{K*}j_X^\dagger \Omega^*_{]Y[_{\widehat{\A}^1_\mathfrak{P}}/]Y[_\mathfrak{P}})
$$
is a quasi-isomorphism as claimed. Since a similar calculation holds when we replace $\mathfrak{P}$ by any open affinoid subset, the claim in relative dimension 1 follows.

In the general case we consider the tower
$$
 (X,Y,\widehat{\A}^d_{\mathfrak{P}}) \overset{u^{(d)}}{\rightarrow}  (X,Y,\widehat{\A}^{d-1}_{\mathfrak{P}}) \overset{u^{(d-1)}}{\rightarrow} \ldots \overset{u^{(1)}}{\rightarrow} (X,Y,\mathfrak{P})
$$
and we know that at each stage, 
$$
j^\dagger _X\cur{O}_{]Y[_{\widehat{\A}^{k-1}_\mathfrak{P}}} \rightarrow \mathbf{R}u^{(k)}_{K*} j_X^\dagger \Omega^*_{]Y[_{\widehat{\A}^k_\mathfrak{P}}/]Y[_{\widehat{\A}^{k-1}_\mathfrak{P}}}
$$
is a quasi-isomorphism. We want to deduce that in fact
$$
j^\dagger _X\Omega^*_{]Y[_{\widehat{\A}^{k-1}\mathfrak{P}}/]Y[_\mathfrak{P}} \rightarrow \mathbf{R}u^{(k)}_{K*} j_X^\dagger \Omega^*_{]Y[_{\widehat{\A}^k_\mathfrak{P}}/]Y[_{\mathfrak{P}}}
$$
is a quasi-isomorphism. To do so, we consider the Gauss--Manin filtration $F^\bu$ on 
$$
j_X^\dagger \Omega^*_{]Y[_{\widehat{\A}^k_\mathfrak{P}}/]Y[_{\mathfrak{P}}}
$$
arising from the composition $]Y[_{\widehat{\A}^{k}_\mathfrak{P}}\rightarrow ]Y[_{\widehat{\A}^{k-1}_\mathfrak{P}} \rightarrow ]Y[_{\mathfrak{P}}$, this is defined by
$$
F^i(j_X^\dagger \Omega^*_{]Y[_{\widehat{\A}^k_\mathfrak{P}}/]Y[_{\mathfrak{P}}}):= \mathrm{im}(j_X^\dagger \Omega^{*-i}_{]Y[_{\widehat{\A}^k_\mathfrak{P}}/]Y[_{\mathfrak{P}}} \otimes  (u_K^{(k)})^*j_X^\dagger \Omega^i_{]Y[_{\widehat{\A}^{k-1}_\mathfrak{P}}/]Y[_{\mathfrak{P}}} \rightarrow j_X^\dagger \Omega^*_{]Y[_{\widehat{\A}^k_\mathfrak{P}}/]Y[_{\mathfrak{P}}}).
$$
Since the terms in the exact sequence
$$ 0 \rightarrow (u_K^{(k)})^*j_X^\dagger \Omega^1_{]Y[_{\widehat{\A}^{k-1}_\mathfrak{P}}/]Y[_{\mathfrak{P}}} \rightarrow j_X^\dagger \Omega^1_{]Y[_{\widehat{\A}^{k}_\mathfrak{P}}/]Y[_{\mathfrak{P}}} \rightarrow j_X^\dagger \Omega^1_{]Y[_{\widehat{\A}^{k}_\mathfrak{P}}/]Y[_{\widehat{\A}^{k-1}_\mathfrak{P}}}\rightarrow 0
$$
are locally free, we can deduce that
$$ \mathrm{Gr}_F^i(j_X^\dagger \Omega^*_{]Y[_{\widehat{\A}^k_\mathfrak{P}}/]Y[_{\mathfrak{P}}}) \cong (u_K^{(k)})^*j_X^\dagger \Omega^i_{]Y[_{\widehat{\A}^{k-1}_\mathfrak{P}}/]Y[_{\mathfrak{P}}} \otimes j_X^\dagger \Omega^{*-i}_{]Y[_{\widehat{\A}^{k}_\mathfrak{P}}/]Y[_{\widehat{\A}^{k-1}_\mathfrak{P}}} 
$$
and hence
\begin{align*}
\mathbf{R}u^{(k)}_{K*}\mathrm{Gr}^i_F(j_X^\dagger \Omega^*_{]Y[_{\widehat{\A}^k_\mathfrak{P}}/]Y[_{\mathfrak{P}}}) 
&\cong  j^\dagger_X\Omega^i_{]Y[_{\widehat{\A}^{k-1}_\mathfrak{P}}/]Y[_\mathfrak{P}} \otimes  \mathbf{R}u^{(k)}_{K*}(j_X^\dagger \Omega^{*-i}_{]Y[_{\widehat{\A}^k_\mathfrak{P}}/]Y[_{\widehat{\A}^{k-1}_\mathfrak{P}}}) \\ 
&\cong  j^\dagger_X\Omega^i_{]Y[_{\widehat{\A}^{k-1}_\mathfrak{P}}/]Y[_\mathfrak{P}}[-i].
\end{align*}
Thus examining the spectral sequence associated to the Gauss--Manin filtration gives
$$ \mathbf{R}u^{(k)}_{K*}(j_X^\dagger \Omega^*_{]Y[_{\widehat{\A}^k_\mathfrak{P}}/]Y[_{\mathfrak{P}}}) \cong j^\dagger_X\Omega^*_{]Y[_{\widehat{\A}^{k-1}_\mathfrak{P}}/]Y[_\mathfrak{P}}
$$
and repeatedly applying this gives
$$ \mathbf{R}u_{K*} (j_X^\dagger \Omega^*_{]Y[_{\widehat{\A}^d_\mathfrak{P}}/]Y[_{\mathfrak{P}}}) \cong j^\dagger_X\cur{O}_{]Y[_\mathfrak{P}}
$$
as required. Again, to deduce the last statement we use the Gauss--Manin filtration on $j_X^\dagger \Omega^*_{]Y[_{\widehat{\A}^d_\mathfrak{P}}/S_K}$ arising from the composition $]Y[_{\widehat{\A}^d_{\mathfrak{P}}} \rightarrow ]Y[_\mathfrak{P}\rightarrow \mathbb{D}^b_K$ in exactly the same way. (Note that local freeness of $j_X^\dagger\Omega^1_{]Y[_\mathfrak{P}/S_K}$ follows from combining smoothness of $\mathfrak{P}$ over $\cur{V}\pow{t}$ in a neighbourhood of $X$ with Propositions \ref{support} and \ref{overconv}.)
\end{proof}

Finally, we will need to know a certain degree of locality on $X$. If $Z=Y\setminus Z$ and $E$ is any sheaf on $]Y[_\mathfrak{P}$ then we define $\underline{\Gamma}_{Z}^\dagger E$ by the exact sequence 
$$ 0 \rightarrow \underline{\Gamma}_{Z}^\dagger E \rightarrow E \rightarrow j_{X}^\dagger E\rightarrow 0.
$$
Note that $j_X^\dagger$ and $\underline{\Gamma}_{Z}^\dagger$ are exact, and we have $j_X^\dagger j_{X'}^\dagger E\cong j_{X\cap X'}^\dagger E $ and $\underline{\Gamma}_{Z}^\dagger\underline{\Gamma}_{Z'}^\dagger E \cong \underline{\Gamma}_{Z\cap Z'}^\dagger E$.

\begin{lem}\label{jdlocal} Let $(X,Y,\mathfrak{P}$) be a $\cur{V}\pow{t}$-frame, and $X=\cup_{j=1}^n X_j$ a finite open cover of $X$. Then for any sheaf $E$ on $]Y[_\mathfrak{P}$ there is an exact sequence of sheaves
$$ 0\rightarrow j_X^\dagger E \rightarrow \prod_j j_{X_j}^\dagger E \rightarrow \prod_{j_0<i_j} j^\dagger_{X_{j_0}\cap X_{j_1}}E \rightarrow \ldots \rightarrow j^\dagger_{\cap_j X_j}E\rightarrow 0
$$
on $]Y[_\mathfrak{P}$.
\end{lem}

\begin{proof} We follow the proof of Proposition 2.1.8 of \cite{iso}, and proceed by induction on the size of the covering $n$. The case $n=1$ is obvious, so assume that $n\geq 2$ and let $X'=\cup_{j=2}^n X_j$. The induction hypothesis implies that
$$ 0\rightarrow j_{X'}^\dagger E \rightarrow \prod_{1<j} j_{X_j}^\dagger E \rightarrow \prod_{1<j_0<i_j} j^\dagger_{X_{j_0}\cap X_{j_1}}E \rightarrow \ldots \rightarrow j^\dagger_{\cap^n_{j=2} X_j}E\rightarrow 0
$$ 
is exact. Let $Z'=Y\setminus X'$ so that the complex
$$ K^\bu := (0\rightarrow E \rightarrow \prod_{1<j} j_{X_j}^\dagger E \rightarrow \prod_{1<j_0<i_j} j^\dagger_{X_{j_0}\cap X_{j_1}}E \rightarrow \ldots \rightarrow j^\dagger_{\cap^n_{j=2} X_j}E\rightarrow 0)
$$ 
is a resolution of $\underline{\Gamma}_{Z'}^\dagger E$. Letting $Z_1=Y\setminus X_1$ we get an exact sequence of complexes
$$ 0 \rightarrow \underline{\Gamma}_{Z_1}^\dagger K^\bu \rightarrow K^\bu \rightarrow j^\dagger_{X_1}K^\bu \rightarrow 0.
$$
We can identify $j_{X_1}^\dagger K^\bu$ with the complex
$$ 0 \rightarrow j_{X_1}^\dagger E \rightarrow \prod_{1<j} j^\dagger_{X_1\cap X_j} E \rightarrow \ldots \rightarrow \ldots j_{\cap_j X_j}^\dagger E \rightarrow 0
$$
and hence we can identify the double complex associated to $K^\bu\rightarrow j_{X_1}^\dagger K^\bu$ with
$$ 0\rightarrow E \rightarrow \prod_{j} j_{X_j}^\dagger E \rightarrow \prod_{j_0<i_j} j^\dagger_{X_{j_0}\cap X_{j_1}}E \rightarrow \ldots \rightarrow j^\dagger_{\cap_j X_j}E\rightarrow 0.
$$ 
This is quasi-isomorphic to $\underline{\Gamma}_{Z_1}^\dagger K^\bu$, which by exactness of $\underline{\Gamma}_{Z_1}^\dagger
$ is in turn quasi-isomorphic to $\underline{\Gamma}_{Z_1}^\dagger\underline{\Gamma}_{Z'}^\dagger E \cong \underline{\Gamma}_{Z}^\dagger
 E$. Thus we have a quasi-isomorphism
 $$
 \underline{\Gamma}_{Z}^\dagger E \cong \left(0\rightarrow E \rightarrow \prod_{j} j_{X_j}^\dagger E \rightarrow \prod_{j_0<i_j} j^\dagger_{X_{j_0}\cap X_{j_1}}E \rightarrow \ldots \rightarrow j^\dagger_{\cap_j X_j}E\rightarrow 0\right)
 $$
 and hence an exact sequence
 $$
 0\rightarrow j_X^\dagger E \rightarrow \prod_{j} j_{X_j}^\dagger E \rightarrow \prod_{j_0<i_j} j^\dagger_{X_{j_0}\cap X_{j_1}}E \rightarrow \ldots \rightarrow j^\dagger_{\cap_j X_j}E\rightarrow 0\
 $$
 as claimed.
\end{proof}

We can now put this all together to prove that, up to isomorphism, the rigid cohomology of $(X,Y,\mathfrak{P})$ only depends on $X$.

\begin{theo} \label{bigindep1} Let 
$$
\xymatrix{
& Y'  \ar[r]^{i'}\ar[d]^v & \mathfrak{P}' \ar[d]^u \\
X \ar[ur]^{j'} \ar[r]^j & Y \ar[r]^i & \mathfrak{P}
}
$$
be a diagram of smooth frames over $\cur{V}\pow{t}$, such that $v$ is proper, and $u$ is smooth in a neighbourhood of $X$. Then the natural maps
$$
H^i_\rig((X,Y,\mathfrak{P})/\ekd)\rightarrow H^i_\rig((X,Y',\mathfrak{P}')/\ekd)
$$
are isomorphisms for all $i\geq0$.
\end{theo}

\begin{proof} We closely follows the proof of Theorem 6.5.2 in \cite{rigcoh}. It suffices to prove that the natural morphism
$$
j^\dagger _X\Omega^*_{]Y[_\mathfrak{P}/S_K} \rightarrow \mathbf{R}u_{K*} j_X^\dagger \Omega^*_{]Y'[_{\mathfrak{P}'}/S_K}
$$
is a quasi-isomorphism, this question is clearly local on $\mathfrak{P}$, and is local on $X$ by Lemma \ref{jdlocal}. Also note that we may at any point replace $Y$ or $Y'$ by closed subschemes containing $X$.

First assume that $v=\mathrm{id}$ is the identity, so we actually have a diagram
$$
\xymatrix{
&  & \mathfrak{P}' \ar[d]^u \\
X  \ar[r]^j & Y \ar[r]^i \ar[ur]^{i'}& \mathfrak{P}
}
$$
with $u$ smooth around $X$. In this case the question is also local on $\mathfrak{P}'$, and hence we may assume that the conclusions of Lemma \ref{frameetale} hold. Hence we may factor $u$ as 
$$
(X,Y,\mathfrak{P}')\overset{u''}{\rightarrow} (X,Y,\widehat{\A}^d_\mathfrak{P}) \overset{u'}{\rightarrow} (X,Y,\mathfrak{P})
$$
where $v$ is \'{e}tale in a neighbourhood of $X$. Hence by Proposition \ref{strongprop} we have
$$
j_X^\dagger\Omega^*_{]Y[_{\widehat{\A}^d_\mathfrak{P}}/S_K}
\cong \mathbf{R}u''_{K*} j_X^\dagger\Omega^*_{]Y[_{\mathfrak{P}'}/S_K},
$$
by Proposition \ref{poincare} we have
$$
j_X^\dagger\Omega^*_{]Y[_\mathfrak{P}/S_K} \cong \mathbf{R}u'_{K*}j_X^\dagger\Omega^*_{]Y[_{\widehat{\A}^d_\mathfrak{P}}/S_K}.
$$
and combining these two then gives the result.

Next assume that $v$ is projective, then exactly as in Lemma 6.5.1 of \emph{loc. cit.} by localising on $\mathfrak{P}$ and $X$, and replacing $Y'$ by some closed subscheme containing $X$ we may assume that we have a morphism of frames 
$$
\xymatrix{
& Y'  \ar[r]^{i''}\ar[d]^v & \mathfrak{P}'' \ar[d]^{u'} \\
X \ar[ur]^{j'} \ar[r]^j & Y \ar[r]^i & \mathfrak{P}
}
$$
with $u'$ \'{e}tale around $X$. We consider the diagram of frames
$$ \xymatrix{  (X,Y',\mathfrak{P}'\times_\mathfrak{P} \mathfrak{P}'') \ar[r]\ar[d] & (X,Y',\mathfrak{P}') \ar[d] \\
(X,Y',\mathfrak{P}'') \ar[r] & (X,Y,\mathfrak{P})  }
$$
and by the case $v=\mathrm{id}$ already proven, we deduce that
$$   \mathbf{R}u'_{K*}j_X^\dagger\Omega^*_{]Y'[_{\mathfrak{P}''} /S_K}\cong  \mathbf{R}u_{K*}j_X^\dagger\Omega^*_{]Y'[_{\mathfrak{P}}/S_K}
$$
since both are isomorphic to $ \mathbf{R}u''_{K*}j_X^\dagger\Omega^*_{]Y'[_{\mathfrak{P}'\times_{\mathfrak{P}}\mathfrak{P}''} /S_K}$, where $u'':\mathfrak{P}'\times_\mathfrak{P}\mathfrak{P}''\rightarrow \mathfrak{P}$ is the canonical map. Again using Proposition \ref{strongprop} we deduce that 
$$  j_X^\dagger\Omega^*_{]Y[_\mathfrak{P}/S_K}\cong  \mathbf{R}u'_{K*}j_X^\dagger\Omega^*_{]Y'[_{\mathfrak{P}''} /S_K}\cong  \mathbf{R}u_{K*}j_X^\dagger\Omega^*_{]Y'[_{\mathfrak{P}'} /S_K}
$$
as required.

Finally we consider the general case. Thanks to Chow's Lemma (see 7.5.13 and 7.5.14 of \cite{chow}) we may blow-up $\mathfrak{P}'$ along a closed subscheme of $Y'$ outside $X$, and obtain a diagram
$$ \xymatrix{ & Y''\ar[d]^{v'} \ar[r] &  \mathfrak{P}'' \ar[d]^{u'}\\
X \ar[r]\ar[ur]\ar[dr] & Y'\ar[d]^v  \ar[r] &  \mathfrak{P}' \ar[d]^u \\
 & Y \ar[r] & \mathfrak{P}
}
$$
where $v\circ v'$ is projective. Since the closed subscheme we are blowing up is contained in $V(\pi)\subset \frak{P}'$, the induced map $u'_K$ is an isomorphism on generic fibres, as well as between tubes. Hence we get 
$$ j_X^\dagger\Omega^*_{]Y'[_{\mathfrak{P}'} /S_K} \cong \mathbf{R}u'_{K*}j_X^\dagger\Omega^*_{]Y''[_{\mathfrak{P}''} /S_K}.
$$
The projective case already proven then implies that 
$$  j_X^\dagger\Omega_{]Y[_\mathfrak{P}/S_K}\cong  \mathbf{R}(u\circ u')_{K*}j_X^\dagger\Omega^*_{]Y''[_{\mathfrak{P}''} /S_K}\cong  \mathbf{R}u_{K*}j_X^\dagger\Omega^*_{]Y'[_{\mathfrak{P}'}/S_K }
$$
as required.
\end{proof}

Of course, in the usual way by considering the fibre product of two frames this implies that the rigid cohomology of any two smooth and proper frames of the form $(X,Y,\mathfrak{P})$ and $(X,Y',\mathfrak{P}')$ are isomorphic. Exactly as in the discussion following Corollaire 1.5 of \cite{finitude}, we then get a functor 
$$
X\mapsto H^i_\rig(X/\ekd)
$$
from the category of embeddable $k\lser{t}$-varieties to $\ekd$-vector spaces. We can thus summarise the results of this section as follows.

\begin{theo} \label{maintheo}There are functors
$$ X\mapsto H^i_\rig(X/\ekd) 
$$
from the category of embeddable varieties over $k\lser{t}$ to vector spaces over $\ekd$, which can be calculated as $H^i(]Y[_\mathfrak{P},j_X^\dagger \Omega^*_{]Y[_\mathfrak{P}})$ for any smooth and proper frame $(X,Y,\mathfrak{P})$. Moreover, the functoriality morphism 
$$
f^*:H^i_\rig(X/\ekd)\rightarrow H^i_\rig(X'/\ekd)
$$
associated to a morphism $f:X'\rightarrow X$ of embeddable varieties can be calculated as that induced by a morphism of smooth and proper frames $(X',Y',\mathfrak{P}')\rightarrow (X,Y,\mathfrak{P})$.
\end{theo}

We will extend this to include coefficients in the next section, and to non-embeddable varieties in the sequel \cite{rclsf3}.

\begin{rem} \label{partial} Actually, we get slightly more, since the proof shows that we can define the rigid cohomology $H^i_\rig((X,Y)/\ekd)$ of any embeddable pair $(X,Y)$ consisting of an open immersion of a $k\lser{t}$-variety into a flat, finite type $k\pow{t}$-scheme. We do so by choosing a closed immersion $Y\rightarrow \mathfrak{P}$ into a finite type formal $\cur{V}\pow{t}$-scheme, smooth over $\cur{V}\pow{t}$ around $X$.

One interesting special case of this is when $Y$ is a compactification of $X$ as a $k\lser{t}$-variety. In this case, we choose an embedding of $Y$ into a finite type formal $\cur{O}_{\ek}$-scheme, smooth around $X$, then such a formal scheme is also of finite type over $\cur{V}\pow{t}$, and smooth over $\cur{V}\pow{t}$ around $X$. Hence $H^i_\rig((X,Y)/\ekd)$ is just the usual rigid cohomology $H^i_\rig(X/\ek)$. More generally, if $Y$ is actually a $k\lser{t}$-variety, then $H^i_\rig((X,Y)/\ekd)$ is just the usual partially overconvergent rigid cohomology $H^i_\rig((X,Y)/\ek)$.

Another interesting special case is when $Y$ is taken to be a model for $X$ over $k\pow{t}$, in this case $H^i_\rig((X,Y)/\ekd)$ is a version of convergent cohomology taking values in $\ekd$ rather than $\ek$. However, since this will not be finite dimensional in general, we see no reason to believe that this should be an $\ekd$-structure on the usual convergent cohomology $H^i_\mathrm{conv}(X/\ek):=H^i_\rig((X,X)/\ek)$.
\end{rem}

\section{Relative coefficients and Frobenius structures}\label{coeffs}

In this section we introduce the coefficients of the cohomology theory $X\mapsto H^i_\rig(X/\ekd)$, namely overconvergent isocrystals (relative to $\ekd$). We follow closely the definition of overconvergent isocrystals given in Chapter 7 of \cite{rigcoh}, which is the inspiration for most of the definitions and results here. The definitions we give will transparently not depend on any choice of a smooth and proper frame containing $X$, however, the key results will be a characterisation in terms of modules with overconvergent connection on a given frame, as well a characterisation of the pullback functor induced by a morphism of varieties in terms of a morphism of frames. We also define cohomology groups with values in an overconvergent isocrystal, which a priori \emph{does} depend on a choice of frame, however, the results of the previous section (or rather, their proofs) will easily imply its independence from such choices. We then discuss Frobenius structures on isocrystals, and introduce the fundamental category of coefficients, the category $F\text{-}\mathrm{Isoc}^\dagger(X/\ekd)$ of overconvergent $F$-isocrystals on a $k\lser{t}$-variety $X$, and give a characterisation in terms of modules with overconvergent connection on a frame, together with a Frobenius structure. Nothing in this section should contain any surprises for those familiar with the theory of rigid cohomology, however, given the novel setting, we thought it best to proceed as slowly and thoroughly as we considered reasonable.

The categories of coefficients that we will consider in this section are \emph{relative} coefficients, that is their differential structure is $\ekd$-linear. This is the set-up most closely linked to classical rigid cohomology, and is also that in which it is perhaps most natural to state and prove the version of the $p$-adic monodromy theorem we will need in order to show finite dimensionality of $\ekd$-valued rigid cohomology for smooth curves. In fact, it is the proof of finite dimensionality for smooth curves in the sequel \cite{rclsf2} that requires us to introduce categories of coefficients, since our eventual strategy will be to push forward via an \'{e}tale map to $\A^1$. However, for our eventual purposes of studying questions such as the weight monodromy conjecture and independence of $\ell$, these coefficients will not be enough. In a sequel \cite{rclsf3}, we will introduce and study categories of `absolute' coefficients, i.e. those for which the connection is relative to $K$, and not $\ekd$. These objects will then come with a natural connection on their cohomology groups, the Gauss--Manin connection, and these groups will therefore become $(\varphi,\nabla)$-modules over $\ekd$. For now, though, we will start with the definition of the categories of `relative' coefficients that we will be interested in, that is overconvergent isocrystals on varieties over $k\lser{t}$ and frames over $\cur{V}\pow{t}$.

\begin{defn} \begin{enumerate}
\item Let $X/k\lser{t}$ be a variety. An $X$-frame over $\cur{V}\pow{t}$ is a frame $(U,W,\mathfrak{Q})$ over $\cur{V}\pow{t}$ together with a $k\lser{t}$-morphism $U\rightarrow X$. A morphism of $X$-frames is a morphism of frames commuting with the given morphism to $X$.
\item An overconvergent isocrystal on $X/\ekd$ is a collection of coherent $j^\dagger_{U}\cur{O}_{]W[_\mathfrak{Q}}$-modules, $\sh{E}_\mathfrak{Q}$, one for each $X$-frame $(U,W,\mathfrak{Q})$, together with isomorphisms $u^*\sh{E}_\mathfrak{Q}\rightarrow \sh{E}_{\mathfrak{Q}'}$ for every morphism of $X$-frames $u:(U',W',\mathfrak{Q}')\rightarrow (U,W,\mathfrak{Q})$, which satisfy the usual cocycle condition. The category of such objects is denoted $\mathrm{Isoc}^\dagger(X/\ekd)$.
\item Let $(X,Y,\frak{P})$ be a $\cur{V}\pow{t}$-frame. A overconvergent isocrystal on $(X,Y,\mathfrak{P})/\ekd$ is a collection of coherent $j^\dagger_{U}\cur{O}_{]W[_\mathfrak{Q}}$-modules, $\sh{E}_\mathfrak{Q}$, one for each frame $(U,W,\mathfrak{Q})$ over $(X,Y,\mathfrak{P})$, together with isomorphisms $u^*\sh{E}_\mathfrak{Q}\rightarrow \sh{E}_{\mathfrak{Q}'}$ for every morphism of frames $u:(U',W',\mathfrak{Q}')\rightarrow (U,W,\mathfrak{Q})$ such that the diagram
$$
\xymatrix{ (U',W') \ar[r]\ar[dr] & (U,W)\ar[d] \\ & (X,Y)
}
$$
commutes, and which satisfy the usual cocycle condition. The category of such objects is denoted $\mathrm{Isoc}^\dagger((X,Y,\mathfrak{P})/\ekd)$.
\end{enumerate}
\end{defn}

These are abelian tensor categories, which admit internal hom objects, and are such that the natural `realisation functors' commute with tensor products (at this stage we do not know flatness of isocrystals, and therefore do not know exactness of the realisations or commutation of realisations with internal hom). There is an obvious functor
$$\mathrm{Isoc}^\dagger(X/\ekd)\rightarrow \mathrm{Isoc}^\dagger((X,Y,\frak{P})/\ekd)
$$
induced by the forgetful functor from frames over $(X,Y,\frak{P})$ to frames over $X$. It is straightforward to verify that the category $\mathrm{Isoc}^\dagger(X/\ekd)$ is local for the Zariski topology on $X$, and that $\mathrm{Isoc}^\dagger((X,Y,\frak{P})/\ekd)$ is local for the Zariski topology on $\frak{P}$. Zariski locality of $\mathrm{Isoc}^\dagger(X/\ekd)$ with respect to $X$ follows from the lemma below.

\begin{lem} \label{colimcoh} Let $(X,Y,\frak{P})$ be a $\cur{V}\pow{t}$-frame. Then restriction followed by push-forward induces an equivalence of categories
$$
\mathrm{colim}_V \mathrm{Coh}(\cur{O}_V) \rightarrow  \mathrm{Coh}(j_X^\dagger\cur{O}_{]Y[_\frak{P}})
$$
where the colimit runs over all open neighbourhoods $V$ of $]X[_\mathfrak{P}$ inside $]Y[_\mathfrak{P}$.
\end{lem}

\begin{proof} Entirely similar to the proof that we will give later on for modules with connection, Lemma \ref{colimcat1}.
\end{proof}

If $u:(X,Y',\mathfrak{P}')\rightarrow (X,Y,\mathfrak{P})$ is a morphism inducing an isomorphism between cofinal systems of neighbourhoods of $]X[_\mathfrak{P}$ in $]Y[_\mathfrak{P}$ and $]X[_{\mathfrak{P}'}$ in $]Y'[_{\mathfrak{P}'}$, then the pullback functor
$$ u^*:\mathrm{Isoc}^\dagger((X,Y,\mathfrak{P})/\ekd)\rightarrow \mathrm{Isoc}^\dagger((X,Y',\mathfrak{P}')/\ekd)
$$
is an equivalence of categories. The first step in interpreting overconvergent isocrystals on an embeddable variety is the following.

\begin{prop} Let $u:(X,Y',\mathfrak{P}')\rightarrow (X,Y,\mathfrak{P})$ be a smooth and proper morphism of frames over $\cur{V}\pow{t}$. Then the pullback functor
$$
u^*:\mathrm{Isoc}^\dagger((X,Y,\mathfrak{P})/\ekd)\rightarrow \mathrm{Isoc}^\dagger((X,Y',\mathfrak{P}')/\ekd)
$$
is an equivalence of categories.
\end{prop}

\begin{proof} The proof, as the proof of the corollary below, goes exactly as in Chapter 7 of \cite{rigcoh}, and is very similar to the proof of Theorem \ref{bigindep1} above. The question is local on $\mathfrak{P}$ and $X$, and we may also at any point replace either $Y$ or $Y'$ by a closed subscheme containing $X$. Again, we divide the proof into three stages.

First assume that the induced map $Y'\rightarrow Y$ is the identity, so that we actually have a diagram
$$
\xymatrix{
&  & \mathfrak{P}' \ar[d]^u \\
X  \ar[r]^j & Y \ar[r]^i \ar[ur]^{i'}& \mathfrak{P}
}
$$
with $u$ smooth around $X$. Here the question is also local on $\mathfrak{P}'$, and hence we may assume that we can factor $u$ as
$$
(X,Y',\mathfrak{P}')\overset{v}{\rightarrow} (X,Y,\widehat{\A}^d_\mathfrak{P}) \overset{w}{\rightarrow} (X,Y,\mathfrak{P})
$$
where $v$ is proper and \'{e}tale. Since $w$ admits a section inducing the identity on $X$, it is formal that it induces an equivalence 
$$
\mathrm{Isoc}^\dagger((X,Y,\mathfrak{P})/\ekd)\isomto \mathrm{Isoc}^\dagger((X,Y,\widehat{\A}^d_\mathfrak{P})/\ekd)
$$
and the fact that $v$ induces an equivalence 
$$
\mathrm{Isoc}^\dagger((X,Y,\widehat{\A}^d_\mathfrak{P})/\ekd)\rightarrow \mathrm{Isoc}^\dagger((X,Y',\mathfrak{P}')/\ekd)
$$
follows from the strong fibration theorem, i.e. Proposition \ref{strongprop}.

Next we assume that $Y'\rightarrow Y$ is projective, hence by localising on $X$ and $\mathfrak{P}$ and replacing $Y'$ if necessary we may assume that we have a morphism of frames 
$$
\xymatrix{
& Y'  \ar[r]^{i''}\ar[d]^v & \mathfrak{P}'' \ar[d]^{u'} \\
X \ar[ur]^{j'} \ar[r]^j & Y \ar[r]^i & \mathfrak{P}
}
$$
with $u'$ \'{e}tale around $X$. We consider the diagram of frames
$$ \xymatrix{  (X,Y',\mathfrak{P}'\times_\mathfrak{P} \mathfrak{P}'') \ar[r]\ar[d] & (X,Y',\mathfrak{P}') \ar[d] \\
(X,Y',\mathfrak{P}'') \ar[r] & (X,Y,\mathfrak{P})  }
$$
and by the case $v=\mathrm{id}$ already proven, together with the strong fibration theorem, we deduce that
\begin{align*}
\mathrm{Isoc}^\dagger((X,Y',\mathfrak{P}')/\ekd) &\cong \mathrm{Isoc}^\dagger((X,Y',\mathfrak{P}'\times_{\mathfrak{P}}\mathfrak{P}'')/\ekd) \\ &\cong \mathrm{Isoc}^\dagger((X,Y',\mathfrak{P}'')/\ekd) \cong \mathrm{Isoc}^\dagger((X,Y,\mathfrak{P})/\ekd)
.
\end{align*}

Finally we consider the general case. As in the proof of Theorem \ref{bigindep1} we may blow up $\mathfrak{P}'$ along a closed subscheme of $Y'$ outside of $X$ to obtain a frame $(X,Y'',\mathfrak{P}'')$ such that $Y''$ is projective over $Y$. Now, since the blow-up induces an isomorphism between a cofinal system of neighbourhoods of $]X[_{\mathfrak{P}'}$ in $]Y'[_{\mathfrak{P}'}$ and $]X[_{\mathfrak{P}''}$ in $]Y''[_{\mathfrak{P}''}$, we therefore have
$$ \mathrm{Isoc}^\dagger((X,Y',\mathfrak{P}')/\ekd) \cong \mathrm{Isoc}^\dagger((X,Y'',\mathfrak{P}'')/\ekd)
$$
But by the projective case already proven, we have 
$$ \mathrm{Isoc}^\dagger((X,Y,\mathfrak{P})/\ekd) \cong \mathrm{Isoc}^\dagger((X,Y'',\mathfrak{P}'')/\ekd)
$$
and the proof is complete.
\end{proof}

\begin{cor} Let $(X,Y,\mathfrak{P})$ be a smooth and proper frame over $\cur{V}\pow{t}$. Then the forgetful functor 
$$
\mathrm{Isoc}^\dagger(X/\ekd)\rightarrow \mathrm{Isoc}^\dagger((X,Y,\mathfrak{P})/\ekd)
$$
is an equivalence of categories.
\end{cor}

We can now interpret $\mathrm{Isoc}^\dagger((X,Y,\mathfrak{P})/\ekd)$ more concretely, first in terms of modules with stratifications, and then in terms of modules with connections. Let us first recall the definitions of stratifications and connections in rigid geometry, as well as the notion of an overconvergent stratification for frames.

For a frame $(X,Y,\mathfrak{P})$ and any $n\geq1$, we let $\mathfrak{P}^n$ denote the $n$-fold fibre product of $\mathfrak{P}$ with itself over $\cur{V}\pow{t}$, the diagonal embedding allows us to consider the frames $(X,Y,\mathfrak{P}^n)$, and the natural inclusions and projections between the different $\mathfrak{P}^n$ induce morphisms of frames
$$
(X,Y,\mathfrak{P}^n)\rightarrow (X,Y,\mathfrak{P}^m)
$$
which we will generally denote by the same letters, e.g. $\Delta$ for the diagonal morphism
$$
(X,Y,\mathfrak{P})\rightarrow (X,Y,\mathfrak{P}^2),
$$ $p_1,p_2$ for the two morphisms of frames
$$
(X,Y,\mathfrak{P}^2)\rightarrow (X,Y,\mathfrak{P}).
$$
and $p_{12},p_{23},p_{13}$ for the projections from the triple product to the double product.

Also, for any smooth rigid variety $\sh{X}$ over $S_K$, we let $\sh{X}^{(n)}$ denote the $n$th infinitesimal neighbourhood of $\sh{X}$ inside $\sh{X}\times_{S_K}\sh{X}$, and $p_i^{(n)}:\sh{X}^{(n)}\rightarrow \sh{X}$ for $i=1,2$ the two natural projections.

\begin{defn} \label{ocrelstrat}Let $(X,Y,\mathfrak{P})$ be a smooth frame over $\cur{V}\pow{t}$.  \begin{enumerate} 
\item An overconvergent stratification on a $j_X^\dagger\cur{O}_{]Y[_\mathfrak{P}}$-module $\sh{E}$ is an isomorphism
$$
\epsilon:p_2^*\sh{E}\rightarrow p_1^*\sh{E}
$$
of $j^\dagger_X\cur{O}_{]Y[_{\mathfrak{P}^2}}$-modules, called the Taylor isomorphism, such that $\Delta^*(\epsilon)=\mathrm{id}$ and $p_{13}^*(\epsilon)=p_{12}^*(\epsilon)\circ p_{23}^*(\epsilon)$ on $]Y[_{\mathfrak{P}^3}$. The category of coherent $j_X^\dagger\cur{O}_{]Y[_\mathfrak{P}}$-modules with overconvergent stratification is denoted $\mathrm{Strat}^\dagger((X,Y,\mathfrak{P})/\ekd)$.
\item A stratification on an $\cur{O}_{]Y[_\frak{P}}$-module $\sh{E}$ is a collection of compatible isomorphisms $$p_2^{(n)*}\sh{E}\isomto p_1^{(n)*}\sh{E}$$ satisfying a cocycle condition similar to that for overconvergent stratifications. The category of coherent $j_X^\dagger\cur{O}_{]Y[_\mathfrak{P}}$-modules with a stratifications as $\cur{O}_{]Y[_\frak{P}}$-modules is denoted $\mathrm{Strat}((X,Y,\mathfrak{P})/\ekd)$.
\item An integrable connection on a $j_X^\dagger\cur{O}_{]Y[_\mathfrak{P}}$-module $\sh{E}$ is a map
$$
\nabla :\sh{E} \rightarrow \sh{E}\otimes \Omega^1_{]Y[_\mathfrak{P}/S_K}
$$
which satisfies the Leibniz rule and is such that the induced map
$$
\nabla^2 :\sh{E} \rightarrow \sh{E}\otimes  \Omega^2_{]Y[_\mathfrak{P}/S_K}
$$
is zero. The category of coherent $j_X^\dagger\cur{O}_{]Y[_\mathfrak{P}}$-modules with an integrable connection is denoted $\mathrm{MIC}((X,Y,\mathfrak{P})/\ekd)$
\end{enumerate}
\end{defn}
Thus in the usual way there is an equivalence of categories
$$\mathrm{Strat}((X,Y,\mathfrak{P})/\ekd) \rightarrow \mathrm{MIC}((X,Y,\mathfrak{P})/\ekd)
$$
and pulling back via the natural morphism
$$ ]Y[_\mathfrak{P}^{(n)}\rightarrow ]Y[_{\mathfrak{P}^2}
$$
gives a functor $$\mathrm{Strat}^\dagger((X,Y,\mathfrak{P})/\ekd) \rightarrow \mathrm{Strat}((X,Y,\mathfrak{P})/\ekd).
$$
Modules with overconvergent stratifications are related to overconvergent isocrystals through the following construction. Given a smooth and proper frame $(X,Y,\mathfrak{P})$ and $\sh{F}\in\mathrm{Isoc}^\dagger((X,Y,\mathfrak{P})/\ekd)$, then we have isomorphisms
$$
p_2^*\sh{F}_\mathfrak{P} \rightarrow \sh{F}_{\mathfrak{P}^2}\leftarrow p_1^*\sh{F}_{\mathfrak{P}}
$$
and hence an isomorphism $p_2^*\sh{F}_\mathfrak{P}\rightarrow p_1^*\sh{F}_\mathfrak{P}$ which satisfies the cocycle condition on $\mathfrak{P}^3$. This induces a functor
$$
\mathrm{Isoc}^\dagger((X,Y,\mathfrak{P})/\ekd)\rightarrow \mathrm{Strat}^\dagger((X,Y,\mathfrak{P})/\ekd)
$$
given by $\sh{F}\mapsto \sh{E}:=\sh{F}_\mathfrak{P}$. This functor is easily checked to be an equivalence. Thus we obtain a series of functors
\begin{align*} \mathrm{Isoc}^\dagger(X/\ekd)\rightarrow \mathrm{Isoc}^\dagger((X,Y,\mathfrak{P})/\ekd)&\rightarrow \mathrm{Strat}^\dagger((X,Y,\mathfrak{P})/\ekd) \\ &\rightarrow \mathrm{Strat}((X,Y,\mathfrak{P})/\ekd) \rightarrow \mathrm{MIC}((X,Y,\mathfrak{P})/\ekd) 
\end{align*}
with everything except $\mathrm{Strat}^\dagger((X,Y,\mathfrak{P})/\ekd) \rightarrow \mathrm{Strat}((X,Y,\mathfrak{P})/\ekd)$ equivalences. We will shortly show that in fact it is fully faithful, however, before we do so we will need the following lemma.

\begin{lem} \label{faithrest}Let $]X[_\frak{P}^\circ$ denote the interior tube of $X$ as in Remark \ref{interior}, this is a rigid space locally of finite type over $\ek$. Then the restriction functor 
$$ \mathrm{Coh}(j_X^\dagger\cur{O}_{]Y[_\frak{P}})\rightarrow \mathrm{Coh}(\cur{O}_{]X[_\frak{P}^\circ})
$$
is faithful.
\end{lem}

\begin{proof} The statement is local on $]Y[_\frak{P}$, so by Lemma \ref{colimcoh} we may replace $]Y[_\frak{P}$ by $[Y]_n$. Suppose that $f:\sh{E}\rightarrow \sh{F}$ is a morphism of coherent $j_X^\dagger\cur{O}_{[Y]_n}$-modules which restrict to zero on $]X[^\circ_\frak{P}\cap [Y]_n$, we may assume that there is a neighbourhood $V$ of $]X[_\frak{P}\cap [Y]_n$ such that $f$ arises from a morphism $f_Vf:E_V\rightarrow F_V$ of coherent $\cur{O}_V$-modules. The support of $\mathrm{im}(f_V)$ is contained in a closed analytic subspace of $V$, and does not meet $]X[^\circ_\frak{P}\cap [Y]_n$. Hence by Proposition \ref{overconv} there must exist a neighbourhood $V'$ of $]X[_\frak{P}\cap [Y]_n$ contained in $V$ such that the support of $\mathrm{im}(f_V)$ does not meet $V'$. Hence $f_V$ is zero on $V'$ and thus $f$ is zero.
\end{proof}

\begin{prop} \label{ffoss} The functor 
$$
\mathrm{Strat}^\dagger((X,Y,\mathfrak{P})/\ekd) \rightarrow \mathrm{Strat}((X,Y,\mathfrak{P})/\ekd).
$$
is fully faithful, with image closed under tensor products and internal hom.
\end{prop}

\begin{proof} Since both categories admit a faithful functor to the category of coherent $j_X^\dagger\cur{O}_{]Y[_\mathfrak{P}}$-modules, it suffices to prove that the functor is full. That is, we must show that if a morphism $\sh{E}\rightarrow \sh{F}$ between modules with overconvergent stratification commutes with the finite level Taylor isomorphisms, then it commutes with the full overconvergent Taylor isomorphism. By looking at the difference between the two natural maps $p_2^*\sh{E}\rightarrow p_1^*\sh{F}$ (and $p_2^{(n)*}\sh{E}\rightarrow p_1^{(n)*}\sh{F}$ for all $n$), it suffices to show that if $\psi:p_2^*\sh{E}\rightarrow p_1^*\sh{F}$ is such that $\psi|_{]Y[_\frak{P}^{(n)}}=0$ for all $n$, then $\psi=0$. Note that this is a question purely about coherent $j_X^\dagger\cur{O}_{]X[_\frak{P}}$-modules.

Let $\frak{P}'$ denote $\frak{P}\otimes_{\cur{V}\pow{t}}\cur{O}_{\ek}$ and $Y'$ denote $Y\otimes_{k\pow{t}}k\lser{t}$. Then Lemma \ref{faithrest} above implies that the restriction functor from coherent $j_X^\dagger\cur{O}_{]Y[_\frak{P}}$-modules to coherent $j_X^\dagger\cur{O}_{]Y'[_{\frak{P}'}}$-modules is faithful. Similarly the restriction functor from $j_X^\dagger\cur{O}_{]Y[_{\frak{P}^2}}$-modules to $j_X^\dagger\cur{O}_{]Y'[_{\frak{P}'^2}}$-modules is faithful, so it suffices to prove the corresponding statement for the tubes over $\ek$, that is if we have a morphism $\psi:p_2^*\sh{E}\rightarrow p_1^*\sh{F}$ between coherent $j_X^\dagger\cur{O}_{]Y'[_{\frak{P}'^2}}$-modules such that $\psi|_{]Y'[_{\frak{P}'}^{(n)}}=0$ for all $n$, then $\psi=0$. But this is just a translation into the language of adic spaces of Lemma 7.2.7 of \cite{rigcoh}.

The statement about tensor product and internal hom is then straightforward.
\end{proof}

\begin{defn} \label{ocrelcon}We say that an integrable connection is overconvergent if it is in the essential image of this functor. The category of modules with overconvergent connections is denoted 
$$
\mathrm{MIC}^\dagger((X,Y,\mathfrak{P})/\ekd) \subset \mathrm{MIC}((X,Y,\mathfrak{P})/\ekd).
$$
\end{defn}
In other words, a connection is overconvergent if the Taylor series converges in a neighbourhood of $]X[_{\mathfrak{P}^2}$. The following theorem summarises the results of this section so far.

\begin{theo} \label{isocmic} Let $X$ be a $k\lser{t}$-variety, and $(X,Y,\mathfrak{P})$ a smooth and proper frame containing $X$. Then the realisation functor $\sh{E}\mapsto \sh{E}_\mathfrak{P}$ induces an equivalence of categories
$$ \mathrm{Isoc}^\dagger(X/\ekd)\rightarrow \mathrm{MIC}^\dagger((X,Y,\mathfrak{P})/\ekd)
$$
from overconvergent isocrystals on $X/\ekd$ to coherent $j_X^\dagger\cur{O}_{]Y[_\mathfrak{P}}$-modules with an overconvergent, integrable connection.
\end{theo}

Of course, this equivalence is natural in the sense that if $u:(X',Y',\mathfrak{P}')\rightarrow (X,Y,\mathfrak{P})$ is a morphism of smooth and proper frames over $\cur{V}\pow{t}$, and $f:X'\rightarrow X$ the induced morphism of $k\lser{t}$-varieties, then the diagram
$$
\xymatrix{  \mathrm{Isoc}^\dagger(X/\ekd) \ar[r]^{f^*}\ar[d] & \mathrm{Isoc}^\dagger(X'/\ekd) \ar[d] \\ 
 \mathrm{MIC}^\dagger((X,Y,\mathfrak{P})/\ekd) \ar[r]^{u^*} & \mathrm{MIC}^\dagger((X',Y',\mathfrak{P}')/\ekd) 
}
$$
commutes up to natural isomorphism. The following particular case of this naturality will be useful.

\begin{cor} Let $u:(X,Y',\mathfrak{P}')\rightarrow (X,Y,\mathfrak{P})$ be a morphism of frames inducing the identity on $X$, and let $\sh{F}\in\mathrm{Isoc}^\dagger(X/\ekd)$ be an overconvergent isocrystal, with realisations $\sh{E}\in\mathrm{MIC}^\dagger((X,Y,\mathfrak{P})/\ekd)$ and $\sh{E}'\in \mathrm{MIC}^\dagger((X,Y',\mathfrak{P}')/\ekd)$. Then $\sh{E}'\cong u^* \sh{E}$ as modules with connection.
\end{cor}

We now define cohomology groups with coefficients in an overconvergent isocrystal. So suppose that $X$ is an embeddable variety and $\sh{E}\in\mathrm{Isoc}^\dagger(X/\ekd)$ is an overconvergent isocrystal. For every smooth and proper frame $(X,Y,\mathfrak{P})$ over $\cur{V}\pow{t}$ we can realise $\sh{E}$ on $(X,Y,\mathfrak{P})$ to give a $j_X^\dagger\cur{O}_{]Y[_\mathfrak{P}}$-module $\sh{E}_\mathfrak{P}$ with an overconvergent integrable connection. Then we define the cohomology groups
$$
H^i_\rig((X,Y,\mathfrak{P})/\ekd,\sh{E}):= H^i(]Y[_\mathfrak{P},\sh{E}\otimes \Omega^*_{]Y[_\mathfrak{P}/S_K}),
$$
as in the constant coefficient case these are vector spaces over $\ekd$. Actually, to check that this is compatible with the definition we gave in the constant case, we will need the following lemma.

\begin{lem} Let $(X,Y,\mathfrak{P})$ be a smooth frame. Then for any $\cur{O}_{]Y[_\mathfrak{P}}$-modules $\cur{E},\cur{F}$ we have a natural isomorphism
$$
j_X^\dagger(\cur{E}\otimes_{\cur{O}_{]Y[_\mathfrak{P}}}\cur{F}) \cong j_X^\dagger\cur{E} \otimes_{\cur{O}_{]Y[_\mathfrak{P}}} \cur{F}
$$
of $j_X^\dagger\cur{O}_{]Y[_\mathfrak{P}}$-modules.
\end{lem}

\begin{proof} First note that since $j_X^\dagger\cur{E}$ is supported on $]X[_\mathfrak{P}$, so is $j_X^\dagger\cur{E} \otimes_{\cur{O}_{]Y[_\mathfrak{P}}} \cur{F}$, and hence it suffices to prove the lemma after applying $j^{-1}$. But now we simply have
\begin{align*} j^{-1}(j_X^\dagger(\cur{E}\otimes_{\cur{O}_{]Y[_\mathfrak{P}}}\cur{F})) &= j^{-1}(\cur{E}\otimes_{\cur{O}_{]Y[_\mathfrak{P}}}\cur{F}) \\
&= j^{-1}\cur{E}\otimes_{j^{-1}\cur{O}_{]Y[_\mathfrak{P}}}j^{-1}\cur{F} \\
&=  j^{-1}j_*j^{-1}\cur{E}\otimes_{j^{-1}\cur{O}_{]Y[_\mathfrak{P}}}j^{-1}\cur{F} \\
&= j^{-1}(j_*j^{-1}\cur{E}\otimes_{\cur{O}_{]Y[_\mathfrak{P}}}\cur{F}) \\
&= j^{-1} (j_X^\dagger\cur{E}\otimes_{\cur{O}_{]Y[_\mathfrak{P}}}\cur{F})
\end{align*}
as required.
\end{proof}
Note that this lemma also implies that rigid cohomology can be calculated as 
$$
H^i_\rig((X,Y,\mathfrak{P})/\ekd,\sh{E}):= H^i(]X[_\mathfrak{P},j^{-1}(\sh{E}\otimes \Omega^*_{]Y[_\mathfrak{P}/S_K})) .
$$

\begin{theo} \label{maintheocoeffs} Up to natural isomorphism $H^i_\rig((X,Y,\mathfrak{P})/\ekd,\sh{E})$ only depends on $X$ and $\sh{E}$ and not on the choice of smooth and proper frame $(X,Y,\mathfrak{P})$. Moreover, the pullback morphism
$$
u^{-1}:H^i_\rig((X,Y,\mathfrak{P})/\ekd,\sh{E})\rightarrow H^i_\rig((X',Y',\mathfrak{P}')/\ekd,f^*\sh{E})
$$
determined by any morphism of frames $u:(X',Y',\mathfrak{P}')\rightarrow (X,Y,\mathfrak{P})$ only depends on the morphism $f:X'\rightarrow X$ induced by $u$. Hence we get a functor
$$
(X,\sh{E}) \rightarrow H^i_\rig(X/\ekd,\sh{E})
$$
in the sense that for any pair of morphisms $f:X'\rightarrow X$ and $f^*\sh{E}\rightarrow \sh{E}'$ there is an induced morphism
$$
H^i_\rig(X/\ekd,\sh{E})\rightarrow H^i_\rig(X'/\ekd,\sh{E}').
$$
\end{theo}

\begin{proof} Identical to the case of constant coefficients $\sh{E}=\cur{O}^\dagger_{X/\ekd}$ treated in the previous section.
\end{proof}

We next show a characterisation of overconvergence of a connection, analogous Theorem 4.3.9 of \cite{rigcoh}. This characterisation will be local, so we will need to know that overconvergence itself is a suitably local property. Clearly $\mathrm{MIC}((X,Y,\mathfrak{P})/\ekd)$ and $\mathrm{MIC}^\dagger((X,Y,\mathfrak{P})/\ekd)$ are local on $\mathfrak{P}$, we will need to know that they are also local on $X$.

We will let $(X,Y,\frak{P})$ be a smooth frame, and for all open neighbourhoods $V$ of $]X[_\mathfrak{P}$ inside $]Y[_\mathfrak{P}$  we will let $ \mathrm{MIC}(\cur{O}_V/S_K)$ denote the category of coherent $\cur{O}_V$-modules with integrable connection relative to $S_K$.

\begin{lem} \label{colimcat1} Restriction followed by push-forward induces an equivalence of categories
$$
\mathrm{colim}_V \mathrm{MIC}(\cur{O}_V/S_K) \rightarrow  \mathrm{MIC}((X,Y,\mathfrak{P})/\ekd)
$$
where the colimit runs over all open neighbourhoods $V$ of $]X[_\mathfrak{P}$ inside $]Y[_\mathfrak{P}$.
\end{lem}

\begin{proof} Let us first prove the corresponding result where we replace $]Y[_\mathfrak{P}$ by the quasi-compact tubes $[Y]_n$, so that the quasi-compact opens $V_{n,m}=[Y]_n\cap U_m$ form a cofinal system of neighbourhoods of $]X[_\mathfrak{P}\cap [Y]_n$ inside $[Y]_n$. By using the internal hom for coherent modules with connection, full faithfulness boils down to showing that for any coherent module with integrable connection $\sh{E}$ on some $V_{n,m}$,
$$ \mathrm{colim}_{m'\geq m}\Gamma(V_{n,m'},\sh{E})^{\nabla=0} \isomto \Gamma([Y]_n,j_*j^{-1}\sh{E})^{\nabla=0}.
$$
Since the natural morphism $$ \mathrm{colim}_{m'\geq m}\Gamma(V_{n,m'},\sh{E}) \rightarrow  \Gamma([Y]_n,j_*j^{-1}\sh{E}) $$ is horizontal, it certainly suffices to prove that this latter morphism is an isomorphism, which follows from Lemma \ref{closed} together with the fact that global sections commute with filtered direct limits on quasi-compact topological spaces. Note that this argument also shows that
$$ \mathrm{colim}_V \mathrm{Coh}(\cur{O}_V) \rightarrow  \mathrm{Coh}(j_X^\dagger\cur{O}_{[Y]_n})
$$
is fully faithful.

To show essential surjectivity, let $\sh{E}$ be some coherent $j_X^\dagger\cur{O}_{[Y]_n}$-module with connection, we first claim that $\sh{E}$ itself comes from a coherent $\cur{O}_V$-module for some $V$. By the full faithfulness for coherent modules, this is local on a finite open covering of $[Y]_n$, hence we may assume that $\sh{E}$ has a presentation. Again using the full faithfulness for coherent modules, this presentation must come from a presentation $$\cur{O}_V^k \rightarrow \cur{O}_V^{k'}\rightarrow E_V\rightarrow 0$$ on some $V$. Now an entirely similar argument to above, using internal hom for abelian sheaves and Lemma \ref{closed} shows that the integrable connection on $\sh{E}$ must come from some integrable connection on $E_V|_{V'}$ for some $V'\subset V$.

We now turn to the original case. So suppose that $\varphi_V:E_V\rightarrow F_V$ is a morphism of coherent $\cur{O}_V$-modules with connection on some neighbourhood $V$ of $]X[_\mathfrak{P}$, such that the induced morphism between coherent $j_X^\dagger\cur{O}_{]Y[_\mathfrak{P}}$-modules is zero. Then for all $n$ the induced morphism of $j_X^\dagger\cur{O}_{[Y]_n}$-modules is zero, and hence there exists some $m$ such that the restriction of $\varphi_V$ to $V\cap V_{n,m}$ is zero. By taking the union over all $n$, there thus exists some sequence $\underline{m}(n)\rightarrow \infty$ such that the restriction of $\varphi_V$ to $V\cap V_{\underline{m}}$ is zero. Hence 
$$
\mathrm{colim}_V \mathrm{MIC}(\cur{O}_V/S_K) \rightarrow  \mathrm{MIC}((X,Y,\mathfrak{P})/\ekd)
$$ is faithful. Similarly, if $E_V$ and $F_V$ are coherent modules with connection on some neighbourhood $V$ of $]X[_\mathfrak{P}$, and $\varphi:\sh{E}\rightarrow \sh{F}$ is a horizontal morphism between the  induced coherent $j_X^\dagger\cur{O}_{]Y[_\mathfrak{P}}$-modules, then for all $n$ we can find some $m=\underline{m}(n)$ such that $V_{n,m}\subset V$ and $\varphi|_{[Y]_n}$ comes from a morphism $\varphi_{n}:E_V|_{V_{n,m}}\rightarrow F_V|_{V_{n,m}}$. By increasing each $m(n)$ in turn, we can ensure that $\varphi_n|_{[Y]_{n-1}}$ agrees with $\varphi_{n-1}$. Hence taking the union over all $n$ gives us a morphism $E_V|_{V_{\underline{m}}}\rightarrow F_V|_{V_{\underline{m}}}$
for some $\underline{m}$. Hence the functor is full.

To show essential surjectivity, suppose that we have some coherent $j_X^\dagger\cur{O}_{]Y[_\mathfrak{P}}$-module with integrable connection $\sh{E}$. Then for all $n$ we know that there exists some $m=m(n)$ such that $\sh{E}|_{[Y]_n}$ comes from some coherent module with connection $E_n$ on $V_{n,m}$ Again, by possibly increasing each $\underline{m}(n)$ in turn, we can ensure that we have isomorphisms $E_{n}|_{V_{n-1,\underline{m}(n-1)}}\cong E_{n-1}$, and hence we can glue the $E_n$ to give a coherent module with connection $E_{\underline{m}}$ on $V_{\underline{m}}$ inducing $\sh{E}$.  \end{proof}

\begin{cor} \label{localconrel} Both the category $\mathrm{MIC}((X,Y,\mathfrak{P})/\ekd)$ and the condition of being overconvergent are local on $X$.
\end{cor}

\begin{proof} Assume that we have an open cover $X=\cup_j X_j$, and compatible objects $\sh{E}_j\in\mathrm{MIC}((X_j,Y,\mathfrak{P})/\ekd)$. By the previous lemma these extend to a compatible collection of coherent $\cur{O}_{V_j}$-modules with connection on some open neighbourhoods $V_j$ of $]X_j[_\mathfrak{P}$ inside $]Y[_\mathfrak{P}$. Hence these glue to to give a coherent module with connection on $V=\cup_j V_j$, which is a neighbourhood of $]X[_\mathfrak{P}$ inside $]Y[_\mathfrak{P}$. An entirely similar argument shows that morphisms glue as well.
Thus $\mathrm{MIC}((X,Y,\mathfrak{P})/\ekd)$ is local on $X$, and since $\mathrm{Strat}^\dagger((X,Y,\mathfrak{P})/\ekd)$ is also local on $X$, so is the overconvergence condition.
\end{proof}

Thus we can test overconvergence locally, and we have the following more concrete criterion. Let $(X,Y,\mathfrak{P})$ be a smooth frame such that $\mathfrak{P}$ is affine and $X=Y\cap D(\bar{g})$ for $\bar{g}$ the reduction of some $g\in\cur{O}_\mathfrak{P}$. Assume further that $\Omega^1_{\mathfrak{P}/\cur{V}\pow{t}}$ has a basis $dt_1,\ldots,dt_n$ in a neighbourhood of $X$, for some functions $t_i\in\cur{O}_\mathfrak{P}$.  Let $(\sh{E},\nabla)\in \mathrm{MIC}((X,Y,\mathfrak{P})/\ekd)$, and let $\partial_i:\sh{E}\rightarrow \sh{E}$ be the derivations corresponding to $dt_i$. For any multi-index $\underline{k}=(k_1,\ldots,k_l)$ we set $\underline{\partial}^{\underline{k}}=\partial_1^{k_1}\ldots\partial_l^{k_l}$.

\begin{prop} \label{ocrel}  Let $V$ be an open neighbourhood of $]X[_\mathfrak{P}$ inside $]Y[_\mathfrak{P}$ such that $dt_1,\ldots,dt_l$ are a basis for $\Omega^1_{V/S_K}$ and $(\sh{E},\nabla)$ arises from a module with integrable connection $(E_V,\nabla_V)$ on $V$. Then $(\sh{E},\nabla)$ is overconvergent if and only if for all $n$, there exists some $m$ and some $d\geq n$ such that $[Y]_d\cap U_m \subset V$, and for every section $e\in \Gamma([Y]_d \cap U_m,E_V)$ we have
$$
\Norm{\frac{\underline{\partial}^{\underline{k}}e}{\underline{k}!}}(r^{-\frac{\norm{\underline{k}}}{n}})\rightarrow 0 
$$
where $\Norm{\cdot}$ is some Banach norm on $ \Gamma([Y]_n \cap U_m,E_V)$.
\end{prop}

\begin{proof} Define $\tau: \frak{P}\times_{\cur{V}\pow{t}} \frak{P}\rightarrow \widehat{\A}^l_{\frak{P}}$ by $\tau_i=p_1^*(t_i)-p_2^*(t_i)$. Let $V_{n,m}=[Y]_n\cap U_m$, be the standard neighbourhoods of $[Y]_n\cap ]X[_\frak{P}$ inside $[Y]_n$, and let $W_{n,m}\subset ]Y[_{\frak{P}^2}$ denote the similar standard neighbourhoods of $]X[_{\frak{P}^2}$ inside $]Y[_{\frak{P}^2}$. By Proposition \ref{strongprop} and its proof, there exists some $d\geq n$ such that for all $m\gg0$, $\tau$ induces an isomorphism $W \cong V_{n,m}\times_{\mathbb{D}^b_K} \mathbb{D}^l_{S_K}(r^{-1/n})$ for some open $W_{n,m} \subset W \subset W_{d,m}$, where $$\mathbb{D}^l_{S_K}(r^{-1/n})=\mathrm{Spa}(S_K\tate{r^{1/n}X_1,\ldots,r^{1/n}X_l})$$ is the polydisc of radius $r^{-1/n}$. Then we have two maps $p_1,p_2:V_{n,m}\times \mathbb{D}^l_{S_K}(r^{-1/n})\rightrightarrows V_{d,m}$, and possibly after increasing $m$ we may assume that $V_{d,m}\subset V$.

If we let $M=\Gamma(V_{d,m},E_V)$, $A=\Gamma(V_{d,m},\cur{O}_{]Y[_\frak{P}})$ and $B=\Gamma(V_{n,m},\cur{O}_{]Y[_\frak{P}})$ then the formal Taylor morphism
$$ E_V|_{V_{d,m}} \rightarrow \varprojlim_n (p_2^{(n)*}E_V|_{V_{n,m}\times \mathbb{D}^l_{S_K}(r^{-1/n})})
$$
can be identified with the map
\begin{align*} M &\rightarrow M\otimes_A B\pow{\tau} \\
e &\mapsto \sum_{\underline{k}} \frac{\underline{\partial}^{\underline{k}}e}{\underline{k}!}\tau^{\underline{k}}
\end{align*}
where $\tau^{\underline{k}}=\tau_1^{k_1}\ldots \tau_l^{k_l}$. Then $\sh{E}$ is overconvergent if and only if we can choose $m$ so that this Taylor series actually converges on $V_{n,m} \times  \mathbb{D}^l_{S_K}(r^{-1/n})$, or in other words, if we can choose $m$ such that 
$$ \sum_{\underline{k}} \frac{\underline{\partial}^{\underline{k}}e}{\underline{k}!}\tau^{\underline{k}} \in M\otimes_A B\tate{r^{1/n}\tau}
$$ 
for all $e$. The proposition follows.
\end{proof}

\begin{cor} \label{subquotext1} For any smooth frame $(X,Y,\frak{P})$ the full subcategory 
$$\mathrm{MIC}^\dagger((X,Y,\mathfrak{P})/\ekd) \subset \mathrm{MIC}((X,Y,\mathfrak{P})/\ekd)
$$ is stable under subobjects and quotients. 
\end{cor}

\begin{proof}  This just follows from the fact that for any map $E_F\rightarrow F_V$ of coherent $\cur{O}_V$-modules, the map
$$
\Gamma([Y]_{n} \cap U_{m},F_V)\rightarrow \Gamma([Y]_{n} \cap U_{m},E_V)
$$
is a strict morphism of Banach $\Gamma([Y]_{n} \cap U_{m},\cur{O}_V)$-modules.\end{proof}

We will also need to know functoriality of coefficients and cohomology under certain extensions of $\ekd$, in particular the following three cases. 

\begin{enumerate} \item The finite extension of $\ekd$ corresponding to a finite separable extension of $k\lser{t}$.
\item The extension determined by some Frobenius lift $\sigma:\ekd\rightarrow \ekd$.
\item The extension $\ekd\rightarrow \ek$.
\end{enumerate}

Firstly, let $F/k\lser{t}$ be a finite separable extension of $k\lser{t}$, $A\subset F$ the ring of integers, and $l/k$ the induced extension of residue fields. Let $L$ be the unramified extension of $K$ lifting $l/k$, with extension $\cur{W}/\cur{V}$ of rings of integers. Let $\cur{E}^{\dagger,F}_K$ and $\cur{E}^F_K$ be the unramified extensions of $\ekd$ and $\ek$ respectively lifting $F/k\lser{t}$, with rings of integers $\cur{O}_{\cur{E}_K^{\dagger,F}}$ and $\cur{O}_{\cur{E}_K^F}$ respectively. These are unique up to non-unique isomorphism, and can be described concretely as follows. Choose a uniformiser $u$ for $F$, so that $F\cong l\lser{u}$. Then we have 
\begin{align*}
\cur{E}_K^F &\cong \left\{\left.\sum_i a_iu^i \in L\pow{u,u^{-1}}\;\right|\;  \sup_i\norm{a_i}<\infty,\;a_i\rightarrow 0\text{ as }i\rightarrow -\infty \right\} \\
\cur{E}_K^{\dagger,F} &\cong \left\{\left.\sum_i a_iu^i \in\cur{E}^F_K\;\right|\;  \exists \eta<1 \text{ s.t. } \norm{a_i}\eta^i\rightarrow 0\text{ as }i\rightarrow -\infty \right\} \\
\cur{O}_{\cur{E}_K^F}&\cong \cur{E}_K^F \cap \cur{W}\pow{u,u^{-1}},\;\; \cur{O}_{\cur{E}_K^{\dagger,F}}\cong \cur{E}_K^{\dagger,F}  \cap \cur{W}\pow{u,u^{-1}}.
\end{align*}
Thus $\cur{W}\pow{u}\subset \cur{E}_K^F$, and we let $S_K^F=\cur{W}\pow{u}\otimes_\cur{W} L\subset \cur{E}_K^F$. Note that although the notation does not suggest so, $S_K^F$ depends on the choice of parameter $u$. Thus the rings
$$(\cur{W}\pow{u},S_K^F)\subset (\cur{O}_{\cur{E}_K^{\dagger,F}}, \cur{E}_K^{\dagger,F})\subset (\cur{O}_{\cur{E}_K^F},\cur{E}_K^F)$$
are of exactly the same form as the rings
$$
(\cur{V}\pow{t},S_K)\subset (\cur{O}_{\ekd},\ekd)\subset (\cur{O}_{\ek},\ek)
$$
but associated to the pair $(L,u)$ rather than the pair $(K,t)$. The base extension $$(k\lser{t},k\pow{t},\cur{V}\pow{t})\rightarrow (F,A,\cur{W}\pow{u})$$ then determines a base change functor
$$
\mathrm{Isoc}^\dagger(X/\ekd)\rightarrow \mathrm{Isoc}^\dagger(X_F/\cur{E}_K^{\dagger,F}),
$$
which we will generally denote by $\sh{E}\mapsto \sh{E}_F$. \emph{A priori}, this construction depends on the choice of parameter $u$. However, since $\cur{O}_{\cur{E}_K^{\dagger,F}}= \mathrm{colim}_m \cur{W}\pow{u}\tate{r^{-1/m}u^{-1}}$ is independent of $u$, one can use the standard neighbourhoods $V'_{n,m}$ of \S\ref{rcolsf} together with Lemma \ref{colimcoh} to show that neither the category $\mathrm{Isoc}^\dagger(X_F/\cur{E}_K^{\dagger,F})$, nor the base extension functor functor 
$$
\mathrm{Isoc}^\dagger(X/\ekd)\rightarrow \mathrm{Isoc}^\dagger(X_F/\cur{E}_K^{\dagger,F}),
$$
nor the cohomology of such objects, as vector spaces over $\cur{E}_{K}^{\dagger,F}$, depend on the choice of $u$. 

To discuss Frobenius structures and pullback, we fix a Frobenius $\sigma_K$ on $K$, that is a field automorphism preserving $\cur{V}$ and lifting the absolute $q$-power Frobenius on $k$.

\begin{defn} A Frobenius on $\cur{V}\pow{t}$ is a $\pi$-adically continuous endomorphism $\sigma_{\cur{V}\pow{t}}:\cur{V}\pow{t}\rightarrow \cur{V}\pow{t}$ which is semi-linear over $\sigma_K$ and lifts the absolute $q$-power Frobenius on $k\pow{t}$.
\end{defn}

Such a Frobenius extends uniquely to a continuous endomorphism of $\cur{O}_{\ek}$, and hence $\ek$, and this preserves the subrings $\cur{O}_{\ekd}$, $\ekd$ and $S_K$. We will henceforth assume that we have chosen a Frobenius on $\cur{V}\pow{t}$, and we endow the rings $S_K,\cur{O}_{\ekd},\ekd,\cur{O}_{\ek}$ and $\ek$ with the induced Frobenii, all of which we will denote by the same letter $\sigma$.

Exactly as above, if we let $X'$ denote the base change of $X$ by the $q$-power Frobenius on $k$, then we get a pullback functor
$$ \sigma^*:\mathrm{Isoc}^\dagger(X/\ekd)\rightarrow \mathrm{Isoc}^\dagger(X'/\ekd)
$$
which we can compose with pullback via relative Frobenius $X\rightarrow X'$, which is $k\lser{t}$-linear, to get a $\sigma$-linear Frobenius pullback functor
$$
F^*:\mathrm{Isoc}^\dagger(X/\ekd)\rightarrow \mathrm{Isoc}^\dagger(X/\ekd).
$$

\begin{defn} \label{maindefs} An overconvergent $F$-isocrystal on $X/\ekd$ is an object $\sh{E}\in \mathrm{Isoc}^\dagger(X/\ekd)$ together with an isomorphism $\varphi:F^*\sh{E}\rightarrow \sh{E}$. The category of overconvergent $F$-isocrystals on $X/\ekd$ is denoted $F\text{-}\mathrm{Isoc}^\dagger(X/\ekd)$.
\end{defn}

\begin{rem} Note that this definition depends on the choice of Frobenius $\sigma$ on $\cur{V}\pow{t}$.
\end{rem}

It is not difficult to see that this construction is compatible with the previous construction associated to a finite separable extension $F/ k\lser{t}$. That is, if we have chosen a Frobenius on $\cur{O}_F$ compatible with that on $\cur{V}\pow{t}$, then this induces a Frobenius pullback on $\mathrm{Isoc}^\dagger(X_F/\cur{E}_K^{\dagger,F})$ and we get a commutative diagram
$$
\xymatrix{ \mathrm{Isoc}^\dagger(X/\ekd) \ar[r] \ar[d]^{F^*} & \mathrm{Isoc}^\dagger(X_F/\cur{E}_K^{\dagger,F}) \ar[d]^{F^*} \\
\mathrm{Isoc}^\dagger(X/\ekd) \ar[r] & \mathrm{Isoc}^\dagger(X_F/\cur{E}_K^{\dagger,F})
}
$$
at least up to natural isomorphism. Thus there is an induced base extension functor 
$$
F\text{-}\mathrm{Isoc}(X/\ekd)\rightarrow F\text{-}\mathrm{Isoc}^\dagger(X_F/\cur{E}_K^{\dagger,F})
$$
which we will again denote by $\sh{E}\mapsto \sh{E}_F$. Again, this is compatible with pullback via morphisms of $k\lser{t}$ varieties $U\rightarrow X$.

Finally, we consider the extension $\ekd\rightarrow \ek$. Let $(X,Y,\mathfrak{P})$ be a smooth and proper frame over $\cur{V}\pow{t}$, and let $(X,Y_{k\lser{t}},\mathfrak{P}_{\cur{O}_{\ek}})$ denote the base change of this frame to $\cur{O}_{\ek}$, this is a smooth and proper frame over $\cur{O}_{\ek}$ in the usual sense of Berthelot's rigid cohomology. Since $\cur{O}_{\ek}\cong \cur{V}\pow{t}\tate{t^{-1}}$, there is a natural open immersion of rigid spaces
$$ ]Y_{k\lser{t}}[_{\mathfrak{P}_{\cur{O}_{\ek}}} \rightarrow ]Y[_\mathfrak{P}
$$
over $S_K$ such that $(j_X^\dagger\cur{O}_{]Y[_\mathfrak{P}})|_{]Y_{k\lser{t}}[_{\mathfrak{P}_{\cur{O}_{\ek}}}}= j_X^\dagger\cur{O}_{]Y_{k\lser{t}}[_{\mathfrak{P}_{\cur{O}_{}\ek}}}$ (which follows, for example, by the concrete description of a cofinal system of neighbourhoods in both cases). This induces a functor
$$
\mathrm{MIC}^\dagger((X,Y,\mathfrak{P})/\ekd)\rightarrow \mathrm{MIC}^\dagger((X,Y_{k\lser{t}},\mathfrak{P}_{\cur{O}_{\ek}})/\ek)
$$
which is simply given by restriction. Here the latter category is the usual category of coherent modules with overconvergent connection as defined for example in Chapter 6 of \cite{rigcoh}. Actually, the definition there is in terms of Tate's rigid spaces rather than Huber's adic spaces, but exactly the same sort of methods as used in Section \ref{opening} will show that the two points of view are equivalent.
The induced functor
$$ \mathrm{Isoc}^\dagger(X/\ekd) \rightarrow \mathrm{Isoc}^\dagger(X/\ek)
$$
is independent of the choice of frame $(X,Y,\mathfrak{P})$ and will be denoted $\sh{E}\mapsto \hat{\sh{E
}}$ (the notation is meant to suggest  a `quasi-completion', that is $\pi$-adic completion in the horizontal variable $t^{-1}$ but not the vertical variables). Again, this is easily seen to be compatible with all previous constructions of Frobenius base change, base change via a finite separable extension of $k\lser{t}$ and pullback via a morphism of $k\lser{t}$-varieties.

All of these `base changes' induce corresponding base change morphisms on cohomology, in that we have canonical base change morphisms
\begin{align*} H^i_\rig(X/\ekd,\cur{E})\otimes_{\ekd}\cur{E}_K^{\dagger,F} &\rightarrow H^i_\rig(X_F /\cur{E}_K^{\dagger,F},\cur{E}_F) \\
H^i_\rig(X/\ekd,\cur{E})\otimes_{\ekd,\sigma}\ekd &\rightarrow H^i_\rig(X /\ekd,F^*\cur{E}) \\
H^i_\rig(X/\ekd,\cur{E})\otimes_{\ekd,\sigma}\ek &\rightarrow H^i_\rig(X /\ek,\hat{\cur{E}}) 
\end{align*}
which are all compatible, in the sense that we leave it to the reader to make precise. In particular, if $\cur{E}\in F\text{-}\mathrm{Isoc}^\dagger(X/\ekd)$ then we get a natural $\sigma$-linear morphism
$$ H^i_\rig(X/\ekd,\cur{E}) \rightarrow H^i_\rig(X /\ekd,\cur{E})
$$
which commutes with the extensions $\ekd\rightarrow \cur{E}_{K}^{\dagger,F}$ and $\ekd\rightarrow \ek$.

We will end this section by noting a couple of easy corollaries of the naturality of Theorem \ref{isocmic}, which give a concrete interpretation of Frobenius pullbacks and Frobenius structures, and will be useful in the sequels \cite{rclsf2,rclsf3}.

\begin{defn} Let $(X,Y,\mathfrak{P})$ be a frame. A Frobenius on $(X,Y,\mathfrak{P})$ is a $\sigma$-linear endomorphism $\varphi$ of $\mathfrak{P}$ lifting the absolute $q$-power Frobenius on $P$. 
\end{defn}

Note that such a Frobenius induces a $\sigma$-linear pullback functor $$\varphi^*:\mathrm{MIC}^\dagger((X,Y,\mathfrak{P})/\ekd)\rightarrow \mathrm{MIC}^\dagger((X,Y,\mathfrak{P})/\ekd),$$
more generally, if $u:(X',Y',\mathfrak{P}')\rightarrow (X,Y,\mathfrak{P})$ is a Frobenius semi-linear morphism of smooth and proper frames over $\cur{V}\pow{t}$ then we get a pullback functor
$$ u^*:\mathrm{MIC}^\dagger((X,Y,\mathfrak{P})/\ekd)\rightarrow \mathrm{MIC}^\dagger((X',Y',\mathfrak{P}')/\ekd)
$$
which is $\sigma$-linear over $\ekd$.

\begin{defn} Let $(X,Y,\mathfrak{P})$ be a frame with Frobenius $\varphi$. Then a Frobenius structure on an object $\sh{E}\in\mathrm{MIC}^\dagger((X,Y,\mathfrak{P})/\ekd)$ is an isomorphism $\varphi^*\sh{E}\isomto \sh{E}$ in $\mathrm{MIC}^\dagger((X,Y,\mathfrak{P})/\ekd)$. The category of modules with an overconvergent integrable connection together with a Frobenius structure is denoted $\varphi\text{-}\mathrm{MIC}^\dagger((X,Y,\mathfrak{P})/\ekd)$.
\end{defn}

\begin{prop}\label{frobisoc} \begin{enumerate} \item Let $u:(X,Y',\mathfrak{P}')\rightarrow (X,Y,\mathfrak{P})$ be a Frobenius semi-linear morphism of smooth and proper frames over $\cur{V}\pow{t}$, such that the induced morphism $X\rightarrow X$ is the absolute $q$-power Frobenius. Then the Frobenius pullback functor
$$
F^*:\mathrm{Isoc}^\dagger(X/\ekd)\rightarrow \mathrm{Isoc}^\dagger(X/\ekd)
$$
can be identified with the functor 
$$ u^*:\mathrm{MIC}^\dagger((X,Y,\mathfrak{P})/\ekd)\rightarrow \mathrm{MIC}^\dagger((X,Y',\mathfrak{P}')/\ekd).
$$
\item Let $(X,Y,\mathfrak{P})$ be a smooth and proper frame over $\cur{V}\pow{t}$ with Frobenius $\varphi$. Then there is an equivalence of categories
$$
F\text{-}\mathrm{Isoc}^\dagger(X/\ekd)\cong \varphi\text{-}\mathrm{MIC}^\dagger((X,Y,\mathfrak{P})/\ekd).
$$
\end{enumerate}
\end{prop}

\bibliographystyle{amsplain}\addcontentsline{toc}{section}{References}
\bibliography{/Users/cdl10/Documents/Dropbox/Maths/lib.bib}

\end{document}